\documentclass[draft]{amsart}
\usepackage{amsmath}
\usepackage{amssymb}
\newtheorem{theorem}{Theorem}[section]
\newtheorem{lemma}[theorem]{Lemma}

\newtheorem{corollary}[theorem]{Corollary}
\theoremstyle{definition}

\newtheorem{example}[theorem]{Example}
\theoremstyle{remark}
\newtheorem{remark}[theorem]{Remark}

\numberwithin{equation}{section}

\begin{document}

\title[The quantization for in-homogeneous self-similar measures]{The quantization for in-homogeneous self-similar measures with in-homogeneous open set condition}
\author{Sanguo Zhu}
\address{School of Mathematics and Physics, Jiangsu University of Technology,
Changzhou 213001, China.} \email{sgzhu@jsut.edu.cn}
\thanks{The author is supported by China Scholarship Council (No. 201308320049).}
\subjclass[2000]{Primary 28A80, 28A78; Secondary 94A15}
\keywords{in-homogeneous self-similar measures, in-homogeneous open set condition, quantization coefficient, quantization dimension.}

\begin{abstract}
Let $(g_i)_{i=1}^M$ be a family of contractive similitudes satisfying the open set condition. Let $\nu$ be a self-similar measure associated with $(g_i)_{i=1}^M$. We study the quantization problem for the in-homogeneous self-similar measure $\mu$ associated with a condensation system $((f_i)_{i=1}^N,(p_i)_{i=0}^N,\nu)$. Assuming a version of in-homogeneous open set condition for this system, we prove the existence of the quantization dimension for $\mu$ of order $r\in(0,\infty)$ and determine its exact value $\xi_r$. We give sufficient conditions for the $\xi_r$-dimensional upper and lower quantization coefficient to be positive or finite.
\end{abstract}

\maketitle

\section{Introduction}

In this paper, we further study the quantization problem for in-homogeneous self-similar measures. This problem has a deep background in information theory and engineering technology \cite{GN:98}. Mathematically, people are concerned with the asymptotic error (quantization error) in the approximation of a given probability measure with discrete ones in the sense of $L_r$-metrics. We refer to \cite{GL:00} for mathematical foundations of the quantization problem and \cite{Las:06,Olsen:07,Olsen:08} for recent results on in-homogeneous self-similar measures. Next, let us recall some definitions and known results.

Let $(f_i)_{i=1}^N$ be a family of contractive similitudes on
$\mathbb{R}^q$ with contraction ratios $(s_i)_{i=1}^N$.
According to \cite{Hu:81}, there exists a
unique non-empty compact subset $E$ of $\mathbb{R}^q$ such that
$E=f_1(E)\cup f_2(E)\cup\cdots\cup f_N(E)$.
The set $E$ is called the self-similar set associated with $(f_i)_{i=1}^N$. Given a probability vector $(q_i)_{i=1}^N$, there exists a unique probability measure $P$ satisfying $P=\sum_{i=1}q_i P\circ f_i^{-1}$. This measure is called the self-similar measure associated with $(f_i)_{i=1}^N$ and $(q_i)_{i=1}^N$. One may see \cite{Hu:81} for more information on self-similar sets and measures. We say that $(f_i)_{i=1}^N$ satisfies the open set condition (OSC) if there exists a bounded non-empty open
set $U$ such that $f_i(U)\cap f_j(U)=\emptyset$ for all $i\neq j$
and $f_i(U)\subset U$ for all $i=1,\cdots,N$.

Let $\nu$ be a Borel probability measure on $\mathbb{R}^q$ with compact support $C$. Let $(p_i)_{i=0}^N$ be a probability vector with $p_i>0$ for all $0\leq i\leq N$. Then by \cite{Bar:88,Las:06}, there exist a unique non-empty compact set $K$ and a unique
Borel probability measure $\mu$ supported on $K$ satisfying
\begin{equation}\label{attractingmeasure}
K=C\cup\bigg(\bigcup_{i=1}^Nf_i(K)\bigg),\;\;\mu=p_0\nu+\sum_{i=1}^Np_i\mu\circ f_i^{-1}.
\end{equation}
Following \cite{Bar:88,Las:06}, we call $((f_i)_{i=1}^N,(p_i)_{i=0}^N,\nu)$ a condensation system. In \cite{Olsen:08}, the measure $\mu$ is called an in-homogeneous self-similar measure with some interesting interpretations for this term. Assuming a version of the open set condition, Olsen and Snigireva studied the Multifractal spectra for such measures $\mu$ (see \cite[Theorem 1.7]{Olsen:08}). A series of basic properties are presented in there on the behavior of the measures $\mu$. Some of these properties will be frequently used in the proof of our preliminary results. In \cite{Zhu:08}, the author has given a characterization for the upper and lower quantization dimension of a class of ISMs, where $\nu$ is a self-similar measure associated with $(f_i)_{i=1}^N$. In there, the existence of the quantization dimension was not proved and the positivity and finiteness of the quantization coefficient have not been examined.

Next, we recall some important objects and known results in quantization theory. One may see \cite{GL:00,GL:01,Kr:08,LM:02,PG:98,PK:01,Za:82} for more information.

For every $n\in\mathbb{N}$, we write $\mathcal{D}_n:=\{\alpha\subset\mathbb{R}^q:1\leq{\rm card}(\alpha)\leq n\}$.
Let $P$ be a Borel probability measure on $\mathbb{R}^q$ and let $r\in(0,\infty)$. The $n$th
quantization error for $P$ of order $r$ is given by \cite{GL:00}
\begin{eqnarray}\label{quanerror}
e_{n,r}(P):=\inf_{\alpha\in\mathcal{D}_{n}}\bigg(\int d(x,\alpha)^{r}dP(x)\bigg)^{\frac{1}{r}}.
\end{eqnarray}
According to \cite{GL:00}, $e_{n,r}(P)$ equals the error in the approximation of $P$ with discrete probability measures supported on at most $n$ points, in the sense of $L_r$-metrics. As a natural characterization of the asymptotic quantization error, the upper and lower quantization dimension for $P$ of order $r$ are defined by
\begin{equation*}
\overline{D}_{r}(P):=\limsup_{n\to\infty}\frac{\log n}{-\log
e_{n,r}(P)};\;\;\underline{D}_{r}(P):=\liminf_{n\to\infty}\frac{\log
n}{-\log e_{n,r}(P)}.
\end{equation*}
If $\overline{D}_{r}(P)=\underline{D}_{r}(P)$, we call the common value the
quantization dimension of $P$ of order $r$ and denote it by $D_r(P)$. In order to obtain more accurate information about the asymptotic quantization error, we are further concerned with the $s$-dimensional upper and lower quantization coefficient (cf. \cite{GL:00,PK:01}):
\begin{eqnarray*}
\overline{Q}_r^s(P):=\limsup_{n\to\infty}n^{\frac{1}{s}}e_{n,r}(P),\;\;
\underline{Q}_r^s(P):=\liminf_{n\to\infty}n^{\frac{1}{s}}e_{n,r}(P),\;\;s>0.
\end{eqnarray*}
By \cite[Proposition 11.3]{GL:00} (see also \cite{PK:01}), the upper (lower) quantization dimension is exactly the critical point at which the upper (lower) quantization coefficient jumps from zero to infinity. Next, let us recall a result of Graf and Luschgy \cite{GL:01}.

Assume that $(f_i)_{i=1}^N$ satisfies the OSC. Let $P$ be the self-similar measure associated with $(f_i)_{i=1}^N$ and a probability vector $(q_i)_{i=1}^N$. Let $k_r$ be the unique solution of the equation $\sum_{i=1}^N(q_is_i^r)^{\frac{k_r}{k_r+r}}=1$. Then
\begin{equation*}
D_r(P)=k_r,\;\;0<\underline{Q}_r^{k_r}(P)\leq\overline{Q}_r^{k_r}(P)<\infty.
\end{equation*}

This result often provides us with significant insight into the study of the quantization problem for non-self-similar probability measures (cf. Theorem \ref{mthm1}).

Before we state our main result of the paper, we need to explain a version of in-homogeneous open set condition (IOSC).
Let ${\rm cl}(A),\partial(A)$ and ${\rm int}(A)$ respectively denote the closure, boundary and interior in $\mathbb{R}^q$ of a set $A$. We will assume the following IOSC: there is a bounded non-empty open set $U$ such that

(A1) $f_i(U)\subset U$ for all $1\leq i\leq N$;

(A2) $f_i(U),1\leq i\leq N$, are pairwise disjoint;

(A3) $E\cap U\neq\emptyset$ and $C\subset U$; where $E$ is the self-similar set associated with $(f_i)_{i=1}^N$;

(A4) $\nu(\partial(U))=0$; $C\cap f_i({\rm cl}(U))=\emptyset$ for all $1\leq i\leq N$.
\begin{remark}
(R1) The IOSC that we assume above is a modified version of the IOSC proposed in \cite[pp.1797]{Olsen:08} by Olsen and Snigireva.
 As the dimensions of a null set do not play any roles in the study of the quantization problem, we have dropped the rest conditions in there. However, to get a better estimate for quantization error of $\mu$, we add the assumption that $f_i({\rm cl}(U))\cap C=\emptyset,1\leq i\leq N$.

(R2) when $M=N$ and $(f_i)_{i=1}^N$ coincides with $(g_i)_{i=1}^M$, we have, $K=E$ and the ISOC is violated in an extreme manner. The mass distribution of $\mu$ is completely different from that of the ISMs considered in the present paper (see Lemma \ref{g1} and Lemma 2 of \cite{Zhu:12}). In fact, one can see that, for the ISMs in \cite{Zhu:12}, the topological supports are simpler, but the mass distributions are more convoluted; while for the ISMs with ISOC, the mass distributions are easier to handle, but  the topological supports are much more complicated. In view of this major difference, the analysis in the following sections is not applicable to the ISMs as considered in \cite{Zhu:12}.
\end{remark}

As our main result of the paper, we will prove that

\begin{theorem}\label{mthm1}
Let $\nu$ be the self-similar measure associated with $(g_i)_{i=1}^M$ satisfying the OSC and a probability vector $(t_i)_{i=1}^M$.  Assume that $((f_i)_{i=1}^N,(p_i)_{i=0}^N,\nu)$ satisfies the IOSC. Then for an ISM $\mu$ as defined in (\ref{attractingmeasure}), $D_r(\mu)$ exists and equals $\xi_r:=\max\{s_r,t_r\}$ and $\underline{Q}_r^{\xi_r}(\mu)>0$, where
$s_r,t_r$ are given by
\begin{equation*}
\sum_{i=1}^M(t_ic_i^r)^{\frac{s_r}{s_r+r}}=1;\;\;\sum_{i=1}^N(p_is_i^r)^{\frac{t_r}{t_r+r}}=1.
\end{equation*}
Moreover, if $s_r>t_r$, then $\overline{Q}_r^{\xi_r}(\mu)<\infty$;
if $s_r=t_r$, then $\underline{Q}^{\xi_r}_r(\mu)=\infty$.
\end{theorem}

As a consequence of Theorem \ref{mthm1}, the following corollary shows that, for sufficiently small $r>0$, the asymptotic property of the quantization error of $\mu$ is essentially identical to that of $\nu$. This reflects some intrinsic properties of the measure $\mu$. Indeed, by \cite[Corollary 2.3]{Olsen:07}, similar property holds for the $L^q$-spectrum of $\mu$.
\begin{corollary}
There exists some $r_0>0$ such that, for all $r\in(0,r_0)$, we have
\begin{eqnarray*}
D_r(\mu)=D_r(v)=s_r,\;\;{\rm and}\;\;0<\underline{Q}_r^{s_r}(\mu)\leq\overline{Q}_r^{s_r}(\mu)<\infty.
\end{eqnarray*}
\end{corollary}
\begin{proof}
By H\"{o}lder's inequality, we have
\begin{eqnarray*}
\sum_{i=1}^N(p_is_i^r)^{\frac{s_r}{s_r+r}}\leq\bigg(\sum_{i=1}^Np_i\bigg)^{\frac{s_r+r}{s_r}}
\bigg(\sum_{i=1}^Nc_i^{s_r}\bigg)^{\frac{r}{s_r+r}}\to (1-p_0),\;\;r\to 0.
\end{eqnarray*}
Thus, there exists some $r_0>0$ such that $s_r>t_r$ for all $0<r<r_0$. The corollary follows immediately from Theorem \ref{mthm1}.
\end{proof}

\section{Notations and preliminary facts}

First we give some notations which we will need later. Set
\begin{eqnarray*}
\Omega_0:=\{\theta\},\;\Omega_k:=\{1,\ldots, N\}^k,\;k\geq 1;\; \Omega^*=\bigcup_{k\geq 0}\Omega_k.
\end{eqnarray*}
For every $k\geq 1$ and $\sigma=(\sigma_1,\ldots,\sigma_k)\in\Omega_k$, we define
\begin{eqnarray*}
f_\sigma:=f_{\sigma_1}\circ f_{\sigma_2}\cdots\circ f_{\sigma_k},\;
p_\sigma:=\prod_{h=1}^np_{\sigma_h},\;s_\sigma:=\prod_{h=1}^ns_{\sigma_h}.
\end{eqnarray*}
For the empty word $\theta$, we define $f_\theta:=1_{\mathbb{R}^d}$ and $p_\theta=s_\theta=1$.

We define $|\sigma|:=n$ for $\sigma\in\Omega_n$. For
any $\sigma\in\Omega^*$ with $|\sigma|\geq n$, we
write $\sigma|_n:=(\sigma_1,\ldots,\sigma_n)$. For $\sigma,\tau\in\Omega^*$, we write
$\sigma\ast\tau:=(\sigma_1,\ldots,\sigma_{|\sigma|},\tau_1,\ldots,\tau_{|\tau|})$. If $\sigma,\tau\in\Omega^*$ and
$|\sigma|\leq|\tau|,\sigma=\tau|_{|\sigma|}$, then we write $\sigma\prec\tau$ and call $\sigma$ a predecessor of $\tau$. Two words
$\sigma,\tau\in\Omega^*$ are said to be incomparable if we have neither $\sigma\prec\tau$
nor $\tau\prec\sigma$. A finite set $\Gamma\subset\Omega^*$ is
called a finite anti-chain if any two words $\sigma,\tau$ in
$\Gamma$ are incomparable. A finite anti-chain is said to be maximal if
any word $\sigma\in\Omega^{\mathbb{N}}$ has a predecessor in $\Gamma$. For a word $\sigma=(\sigma_1,\ldots,\sigma_n)\in\Omega_n$, we define
\begin{equation}\label{s1}
\sigma^-=\theta \;\;{\rm if}\;\;n=1,\;\;\sigma^-:=\sigma|_{n-1}=(\sigma_1,\ldots,\sigma_{n-1})\;\;{\rm if}\;\;n>1.
\end{equation}
For every $\sigma\in\Omega^*$ and $h\geq 1$, we will need to consider:
\begin{eqnarray}\label{s30}
\Gamma(\sigma,h):=\{\tau\in\Omega_{|\sigma|+h}:\sigma\prec\tau\},\;\;\Gamma^*(\sigma):=\bigcup_{k\geq 1}\Gamma(\sigma,h).
\end{eqnarray}
For every $n\geq1$, by iterating (\ref{attractingmeasure}), one easily gets (cf. \cite{Olsen:08})
\begin{eqnarray}
K=\bigg(\bigcup_{h=0}^{n-1}\bigcup_{\sigma\in\Omega_h}f_\sigma(C)\bigg)\cup\bigg(\bigcup_{\sigma\in\Omega_n}f_\sigma(K)\bigg),\label{s2}\\
\mu=p_0\sum_{h=0}^{n-1}\sum_{\sigma\in\Omega_h}p_\sigma\nu\circ f_\sigma^{-1}+\sum_{\sigma\in\Omega_n}p_\sigma\mu\circ f_\sigma^{-1}.\label{s3}
\end{eqnarray}

Let $\mathcal{B}$ denote the Borel $\sigma$-algebra of $\mathbb{R}^q$. We will need to apply the following results of Olsen and Snigireva. We collect these results in the following lemma. One may see Lemmas 4.1, 4.2 and Proposition 4.3 of \cite{Olsen:08} for more details.
\begin{lemma}\label{olsen}\cite{Olsen:08}
Let $U$ be the open set from the IOSC. Then, we have
\begin{enumerate}
\item[\rm (a)] ${\rm supp}(\mu)=K\subset {\rm cl}(U)$;
\item[\rm (b)]$f_{\tau}^{-1}f_\sigma({\rm cl}(U))=\emptyset$ for $\sigma,\tau\in\Omega^*$ with $|\sigma|=|\tau|$; $\nu(f_\tau^{-1}f_\sigma({\rm cl}(U)))=0$ for $\sigma,\tau\in\Omega^*$ with $|\tau|<|\sigma|$;
\item[\rm (c)]$\mu(B)=\mu(B\cap U),B\in\mathcal{B}$;
\item[\rm (d)]$\mu(f_\sigma(K))=\mu(f_\sigma(\rm cl(U)))=p_\sigma,\;\sigma\in\Omega^*$;
\end{enumerate}
\end{lemma}
\begin{remark}{\rm
As is shown by the proofs in \cite[Section 4]{Olsen:08}, Lemma \ref{olsen} holds if the following conditions are satisfied: (l1) $f_i(U)\subset U$ for all $1\leq i\leq N$; (l2) $f_i(U),1\leq i\leq N$, are pairwise disjoint; (l3) $E\cap U\neq\emptyset$; (l4) $\nu(\partial(U))=0$; $\nu(f_i({\rm cl}(U))=0$ for all $1\leq i\leq N$. One can see that these conditions are ensured by the IOSC in (A1)-(A4).
}\end{remark}

Let $\nu$ be the self-similar measure associated with $(g_i)_{i=1}^M$ satisfying the OSC and a probability vector $(t_i)_{i=1}^M$. Let $c_i$ be the contraction ratio of $g_i$ for $i=1,\ldots, M$. Set
\begin{eqnarray*}
\Phi_0:=\{\theta\},\;\;\Phi_k:=\{1,\ldots, M\}^k,\;\;\Phi^*=\bigcup_{k\geq 0}\Phi_k.
\end{eqnarray*}
For $\omega\in\Phi^*$, let $t_\omega,d_\omega$ be defined in the same way as we did for $p_\sigma$ in (\ref{s1}). We define $g_\omega$ the same way as $f_\sigma$ and write $C_\omega:=g_\omega(C)$. For every pair $\omega,\rho\in\Phi^*$, let $\omega\ast\rho$ be defined as we did for words of $\Omega^*$. We have

\begin{lemma}\label{g1}
Let $\mu$ be as stated in Theorem \ref{mthm1}. Assume that the IOSC is satisfied. Then for every $\sigma\in\Omega^*$ and $\omega\in\Phi^*$, we have
$\mu(f_\sigma(C_\omega))=p_0p_\sigma t_\omega$.
\end{lemma}
\begin{proof}
By (A4), we have, $C\cap f_i(U)=\emptyset$ for all $1\leq i\leq N$. Hence,
\begin{eqnarray*}
f_i^{-1}(C_\omega)\cap U=f_i^{-1}(C_\omega\cap f_i(U))\subset f_i^{-1}(C\cap f_i(U))=\emptyset.
\end{eqnarray*}
So $\mu(f_i^{-1}(C_\omega)\cap U)=0$ for all $1\leq i\leq N$. Note that $\nu$ is the self-similar measure associated with $(g_i)_{i=1}^M$ satisfying the OSC and the probability vector $(t_i)_{i=1}^M$. Thus, $\nu(C_\omega)=t_\omega$ for every $\omega\in\Phi^*$ (cf. \cite{Gr:95}). By (\ref{attractingmeasure}) and Lemma \ref{olsen} (c), we deduce
\begin{eqnarray*}
\mu(C_\omega)&=&p_0\nu(C_\omega)+\sum_{i=1}^Np_i\mu\circ f_i^{-1}(C_\omega)\\&=&p_0t_\omega+\sum_{i=1}^Np_i\mu(f_i^{-1}(C_\omega)\cap U)=p_0t_\omega.
\end{eqnarray*}
 By Lemma \ref{olsen} (b), for all $\tau\in\Omega^*$ with $|\tau|<|\sigma|$, we have
\begin{eqnarray*}
\nu\circ f_\tau^{-1}(f_\sigma(C_\omega))\leq\nu\circ f_\tau^{-1}(f_\sigma({\rm cl}(U)))=0;
\end{eqnarray*}
for all $\tau\in\Omega_{|\sigma|}$ with $\tau\neq\sigma$, by Lemma \ref{olsen} (b), we have $f_\tau^{-1}f_\sigma({\rm cl}(U))\cap U=\emptyset$. Hence,
\begin{eqnarray*}
\mu\circ f_\tau^{-1}(f_\sigma(C_\omega))=\mu(f_\tau^{-1}(f_\sigma(C_\omega))\cap U)\leq\mu(f_\tau^{-1}(f_\sigma({\rm cl}(U))\cap U)=0.
\end{eqnarray*}
Thus, by applying (\ref{s3}) with $n=|\sigma|$, we have
\begin{eqnarray*}
\mu(f_\sigma(C_\omega))&=&p_0\sum_{h=0}^{|\sigma|-1}\sum_{\tau\in\Omega_h}p_\sigma\nu\circ f_\tau^{-1}(f_\sigma(C_\omega))\\&&\;\;\;\;+\sum_{\tau\in\Omega_{|\sigma|}}p_\tau\mu\circ f_\tau^{-1}(f_\sigma(C_\omega))\\&=&p_\sigma\mu\circ f_\sigma^{-1}(f_\sigma(C_\omega))=p_\sigma\mu(C_\omega)=p_0p_\sigma t_\omega.
\end{eqnarray*}
This completes the proof of the lemma.
\end{proof}

In order to estimate the quantization error of $\mu$, we need to divide its support $K$ into small parts with some kind of uniformity. Our next lemma is the first step of doing so. For a finite maximal antichain $\Gamma$, we define
\begin{eqnarray*}
l(\Gamma):=\min_{\sigma\in\Gamma}|\sigma|,\;\;L(\Gamma):=\max_{\sigma\in\Gamma}|\sigma|.
\end{eqnarray*}
For each $\sigma\in\Omega_{l(\Gamma)}$, we define
\begin{eqnarray*}
\Lambda_\Gamma(\sigma):=\{\tau\in\Omega^*:\sigma\prec\tau,\Gamma^*(\tau)\cap\Gamma\neq\emptyset\},\;\;\Lambda_\Gamma^*:=\bigcup_{\sigma\in\Omega_{l(\Gamma)}}\Lambda_\Gamma(\sigma).
\end{eqnarray*}
By the definition of $\Gamma^*(\sigma)$ (see (\ref{s30})), we note that if $l(\Gamma)=L(\Gamma)$, then $\Lambda_\Gamma^*=\emptyset$.
\begin{lemma}\label{g2}
Let $\Gamma$ be a finite maximal antichain. Then
\begin{eqnarray}\label{s5}
K=\bigg(\bigcup_{h=0}^{l(\Gamma)-1}\bigcup_{\sigma\in\Omega_h}f_\sigma(C)\bigg)\cup\bigg(\bigcup_{\sigma\in\Lambda_\Gamma^*}f_\sigma(C)\bigg)
\cup\bigg(\bigcup_{\sigma\in\Gamma}f_\sigma(K)\bigg).
\end{eqnarray}
\end{lemma}
\begin{proof}
We give the proof by mathematical induction with $L(\Gamma)$. By (\ref{s2}),
\begin{eqnarray*}
K=\bigg(\bigcup_{h=0}^{l(\Gamma)-1}\bigcup_{\sigma\in\Omega_h}f_\sigma(C)\bigg)\cup\bigg(\bigcup_{\sigma\in\Omega_{l(\Gamma)}}f_\sigma(K)\bigg).
\end{eqnarray*}
So, if $L(\Gamma)=l(\Gamma)$, the lemma is true. Now assume that (\ref{s5}) holds for $L(\Gamma)=l(\Gamma)+p$. Next we show that it holds for $L(\Gamma)=l(\Gamma)+p+1$. Let $\Gamma$ be such a finite maximal antichain. We write
\begin{eqnarray*}
\Gamma^{(-1)}:=\big\{\tau^-:\tau\in\Gamma\cap\Omega_{L(\Gamma)}\big\},\;\Gamma_\flat:=\big(\Gamma\setminus\Omega_{L(\Gamma)}\big)\cup\Gamma^{(-1)}.
\end{eqnarray*}
Then $l(\Gamma_\flat)=l(\Gamma)$ and $L(\Gamma_\flat)=l(\Gamma_\flat)+p$. By the induction hypothesis, we have
\begin{eqnarray}\label{s6}
K=\bigg(\bigcup_{h=0}^{l(\Gamma)-1}\bigcup_{\sigma\in\Omega_h}f_\sigma(C)\bigg)\cup
\bigg(\bigcup_{\omega\in\Lambda_{\Gamma_\flat}^*}f_\omega(C)\bigg)\cup\bigg(\bigcup_{\sigma\in\Gamma_\flat}f_\sigma(K)\bigg).
\end{eqnarray}
For every $\sigma\in\Gamma^{(-1)}$, by (\ref{attractingmeasure}), we have
\begin{eqnarray}\label{s31}
f_\sigma(K)=f_\sigma(C)\cup f_\sigma\bigg(\bigcup_{i=1}^Nf_i(K)\bigg)=f_\sigma(C)\cup \bigg(\bigcup_{i=1}^Nf_{\sigma\ast i}(K)\bigg).
\end{eqnarray}
Note that $\Gamma^{(-1)}\cup\Lambda_{\Gamma_\flat}^*=\Lambda_\Gamma^*$ and that $\big(\Gamma_\flat\setminus\Gamma^{(-1)}\big)\cup\big(\bigcup_{\sigma\in\Gamma^{(-1)}}\Gamma(\sigma, 1)\big)=\Gamma$. Using (\ref{s6}) and (\ref{s31}),
we deduce
\begin{eqnarray*}
K&=&\bigg(\bigcup_{h=0}^{l(\Gamma)-1}\bigcup_{\sigma\in\Omega_h}f_\sigma(C)\bigg)\cup
\bigg(\bigcup_{\omega\in\Lambda_{\Gamma_\flat}^*}f_\omega(C)\bigg)\\&&\cup
\bigg(\bigcup_{\sigma\in\Gamma_\flat\setminus\Gamma^{(-1)}}f_\sigma(K)\bigg)\cup\big(\bigcup_{\sigma\in\Gamma^{(-1)}}f_\sigma(K)\big)\\
&=&\bigg(\bigcup_{h=0}^{l(\Gamma)-1}\bigcup_{\sigma\in\Omega_h}f_\sigma(C)\bigg)\cup
\bigg(\bigcup_{\omega\in\Lambda_{\Gamma_\flat}^*}f_\omega(C)\bigg)\cup
\bigg(\bigcup_{\sigma\in\Gamma_\flat\setminus\Gamma^{(-1)}}f_\sigma(K)\bigg)\\&&
\cup\bigg(\bigcup_{\sigma\in\Gamma^{(-1)}}f_\sigma(C)\bigg)\cup\bigg(\bigcup_{\sigma\in\Gamma^{(-1)}}\bigcup_{\tau\in\Gamma(\sigma, 1)}f_\tau(K)\bigg)\\&=&\bigg(\bigcup_{h=0}^{l(\Gamma)-1}\bigcup_{\sigma\in\Omega_h}f_\sigma(C)\bigg)\cup\bigg(\bigcup_{\omega\in\Lambda_\Gamma^*}f_\omega(C)\bigg)
\cup\bigg(\bigcup_{\sigma\in\Gamma}f_\sigma(K)\bigg).
\end{eqnarray*}
This completes the proof of the lemma.
\end{proof}
\section{Proof of Theorem \ref{mthm1}}

Let $|A|$ denote the diameter of a set $A$. Without loss of generality, we assume that $|K|=1$, so that $|f_\sigma(K)|=s_\sigma$ for every $\sigma\in\Omega^*$. Set
\begin{eqnarray*}
\underline{\eta}_r:=\min\bigg\{\min_{1\leq i\leq N}p_is_i^r,\min_{1\leq i\leq M}t_ic_i^r\bigg\}.
\end{eqnarray*}
We will need the following finite maximal antichains in $\Omega^*$:
\begin{equation}\label{s7}
\Gamma_{k,r}:=\{\sigma\in\Omega^*:p_{\sigma^-}s_{\sigma^-}^r\geq \underline{\eta}_r^k>p_\sigma c_\sigma^r\},\;k\geq 1.
\end{equation}
For simplicity of notations, we write
\begin{eqnarray*}
l_{1,k}:=l(\Gamma_{k,r}),\;\;l_{2,k}:=L(\Gamma_{k,r}),\;N_{k,r}:={\rm card}(\Gamma_{k,r}),\;\;k\geq 1.
\end{eqnarray*}
By Lemma \ref{g2}, we have
\begin{eqnarray*}
K=\bigg(\bigcup_{h=0}^{l_{1,k}-1}\bigcup_{\sigma\in\Omega_h}f_\sigma(C)\bigg)\cup\bigg(\bigcup_{\omega\in\Lambda_{\Gamma_{k,r}}^*}f_\omega(C)\bigg)
\cup\bigg(\bigcup_{\sigma\in\Gamma_{k,r}}f_\sigma(K)\bigg).
\end{eqnarray*}
Note that, $\bigcup_{\sigma\in\Gamma_{k,r}}f_\sigma(K)$ is a proper subset of $K$. We will construct suitable coverings and subsets of $K$ according to $\Gamma_{k,r}$. For convenience, we write
\begin{eqnarray}\label{s8}
\Psi_{k,r}:=\bigcup_{h=0}^{l_{1,k}-1}\Omega_h\cup\Lambda_{\Gamma_{k,r}}^*,\;k\geq 1.
\end{eqnarray}
\begin{lemma}\label{g3}
For every $\sigma\in\Psi_{k,r}$, we have
\begin{eqnarray*}
p_\sigma s_\sigma^r\geq \underline{\eta}_r^k\;\;{\rm for\;all}\;\;\sigma\in\Psi_{k,r}.
\end{eqnarray*}
\end{lemma}
\begin{proof}
For every $0\leq h\leq l_{1,k}-1$ and every $\sigma\in\Omega_h$, we have, $p_\sigma s_\sigma^r\geq \underline{\eta}_r^k$; otherwise, we would have that $l(\Gamma_{k,r})<l_{1,k}$, a contradiction.

For every $\sigma\in\Lambda_{\Gamma_{k,r}}^*$, by the definition, there exists $\tau\in\Gamma_{k,r}$ with $\sigma\prec\tau$ and $|\tau|<|\sigma|$.
Hence, $p_\sigma s_\sigma^r\geq p_{\tau^-}s_{\tau^-}^r\geq \underline{\eta}_r^k$.
The lemma follows.
\end{proof}

As the second step of dividing $K$, for each $\sigma\in\Psi_{k,r}$, we divide $f_\sigma(C)$ into small parts $f_\sigma(C_\rho)$ by means of some finite maximal antichain $\Gamma_{k,r}(\sigma)$ in $\Phi^*$. More exactly, for each $\sigma\in\Psi_{k,r}$, we define
\begin{eqnarray}\label{s11}
\Gamma_{k,r}(\sigma):=\{\rho\in\Phi^*:p_\sigma s_\sigma^r t_{\omega^-}c_{\omega^-}^r\geq\underline{\eta}_r^k>p_\sigma s_\sigma^r t_\omega c_\omega^r\},
\end{eqnarray}
where $\omega^-:=\omega|_{|\omega|-1}$ is defined in the same way as we did for words in $\Omega^*$. Note that $C$ is the self-similar set associated with $(g_i)_{i=1}^M$. We have
\begin{eqnarray}\label{s19}
C=\bigcup_{i=1}^Mg_i(C)=\bigcup_{\rho\in\Gamma_{k,r}(\sigma)}g_\rho(C)=\bigcup_{\rho\in\Gamma_{k,r}(\sigma)} C_\rho.
\end{eqnarray}
Thus, by Lemma \ref{g2} and (\ref{s19}), we are able to divide $K$ in the following manner:
\begin{eqnarray}\label{s20}
K&=&\bigg(\bigcup_{h=0}^{l_{1,k}-1}\bigcup_{\sigma\in\Omega_h}f_\sigma(C)\bigg)\cup\bigg(\bigcup_{\sigma\in\Lambda_{\Gamma_{k,r}}^*}f_\omega(C)\bigg)
\cup\bigg(\bigcup_{\sigma\in\Gamma_{k,r}}f_\sigma(K)\bigg)\nonumber\\
&=&\bigg(\bigcup_{\sigma\in\Psi_{k,r}}\bigcup_{\rho\in\Gamma_{k,r}(\sigma)}f_\sigma(C_\rho)\bigg)\cup\bigg(\bigcup_{\sigma\in\Gamma_{k,r}}f_\sigma(K)\bigg).
\end{eqnarray}
For every $\sigma\in\Psi_{k,r}$, let $M_{k,r}(\sigma)$ denote the cardinality of $\Gamma_{k,r}(\sigma)$. We define
\begin{eqnarray*}
\phi_{k,r}:=N_{k,r}+\sum_{\sigma\in\Psi_{k,r}}M_{k,r}(\sigma).
\end{eqnarray*}
For the purpose of estimating of the upper (lower) quantization coefficient, we need to compare $\phi_{k,r}$ and $\phi_{k+1,r}$. We have
\begin{lemma}\label{g5}
There exists a constant $d_1>0$ such that
\begin{equation*}
\phi_{k,r}\leq\phi_{k+1,r}\leq d_1\phi_{k,r},\;k\geq 1.
\end{equation*}
\end{lemma}
\begin{proof}
For convenience, we set
\begin{eqnarray*}
\overline{\eta}_r:=\max\big\{\max_{1\leq i\leq N} p_is_i^r,\max_{1\leq i\leq M} t_ic_i^r\big\},\;\;H:=\min\{h\in\mathbb{N}:\overline{\eta}_r^h<\underline{\eta}_r\}.
\end{eqnarray*}
Then, we have the following facts.

(f1) for every $\sigma\in\Gamma_{k,r}$, we have $p_{\sigma^-}s_{\sigma^-}^r\geq\underline{\eta}_r^k>p_\sigma s_\sigma^r$. Hence, for all $\tau\in\Omega_H$,
\begin{eqnarray*}
p_{\sigma\ast\tau}s_{\sigma\ast\tau}^r<p_\sigma s_\sigma^r\underline{\eta}_r<\underline{\eta}_r^{k+1}.
\end{eqnarray*}
This implies that $N_{k,r}\leq N_{k+1,r}\leq N^{H}N_{k,r}$ and $l_{1,k}\leq l_{1,k+1}\leq l_{1,k}+H$.

(f2) Let $\sigma\in\Psi_{k,r}$ and $\rho\in\Gamma_{k,r}(\sigma)$. By the definition, we have
\begin{eqnarray*}
p_\sigma c_\sigma^rt_{\rho^-}c_{\rho^-}^r\geq \underline{\eta}_r^k>p_\sigma c_\sigma^r t_\rho c_\rho^r.
\end{eqnarray*}
Hence, for every $1\leq i\leq M$, $p_\sigma c_\sigma^r t_{\rho\ast i} c_{\rho\ast i}^r\geq \underline{\eta}_r^{k+1}$. For any $\omega\in\Phi_H$, we have
\begin{eqnarray*}
p_\sigma s_\sigma^rt_{\rho\ast\omega}c_{\rho\ast\omega}^r<p_\sigma s_\sigma^rt_\rho c_\rho^r\underline{\eta}_r<\underline{\eta}_r^{k+1}.
\end{eqnarray*}
Hence, we have $M_{k,r}(\sigma)\leq M_{k+1,r}(\sigma)\leq M^HM_{k,r}(\sigma)$.

(f3) For every $l_{1,k}\leq h\leq l_{k+1,1}-1$ and $\sigma\in\Omega_h$, there is a unique $\tau\in\Omega_{l_{1,k}-1}$ such that $\tau\prec\sigma$. It is clear that $p_\sigma s_\sigma^r<p_\tau c_\tau^r$. Therefore, for every $\rho\in\Gamma_{k+1,r}(\sigma)$, we have
\begin{eqnarray*}
p_\tau s_\tau^r t_{\rho^-} c_{\rho^-}^r\geq p_\sigma s_\sigma^r t_{\rho^-} c_{\rho^-}^r\geq \underline{\eta}_r^{k+1}>p_\sigma s_\sigma^r t_\rho c_\rho^r.
\end{eqnarray*}
This implies that, either $\rho\in\Gamma_{k+1,r}(\tau)$, or more than one descendant of $\rho$ are contained in $\Gamma_{k+1,r}(\tau)$. From this we deduce that $M_{k+1,r}(\sigma)\leq M_{k+1,r}(\tau)$. Thus, let $D:=\sum_{i=1}^{H} M^h$, using (b2), we deduce
\begin{eqnarray*}
&&\sum_{h=l_{1,k}}^{l_{1,k+1}-1}\sum_{\sigma\in\Omega_h,\tau\prec\sigma}M_{k+1,r}(\sigma)\leq \sum_{i=1}^{l_{1,k+1}-l_{1,k}} M^iM_{k+1,r}(\tau)\\&&\leq\sum_{i=1}^{H} M^i M_{k+1,r}(\tau)=D M_{k+1,r}(\tau)\leq DM^HM_{k,r}(\tau).
\end{eqnarray*}
Note that the preceding inequality holds for every $\tau\in\Omega_{l_{1,k}-1}$. Hence,
\begin{eqnarray*}
\sum_{h=l_{1,k}}^{l_{1,k+1}-1}\sum_{\sigma\in\Omega_h}M_{k+1,r}(\sigma)\leq DM^H\sum_{\tau\in\Omega_{l_{1,k}}}M_{k,r}(\tau)
\end{eqnarray*}

(f4) Let $\sigma\in\Lambda_{\Gamma_{k+1,r}}^*\setminus\Lambda_{\Gamma_{k,r}}^*$, there exists a $\tau\in\Lambda_{\Gamma_{k,r}}^*$ such that $\tau\prec\sigma$ and $|\sigma|<|\tau|+H$.
As we did in (b3), one easily gets
\begin{eqnarray*}
\sum_{\sigma\in\Lambda_{\Gamma_{k+1,r}}^*\setminus\Lambda_{\Gamma_{k,r}}^*,\tau\prec\sigma}M_{k+1,r}(\sigma)\leq D M_{k+1,r}(\tau)\leq DM^HM_{k,r}(\tau).
\end{eqnarray*}
Taking those $\sigma\in\Lambda_{\Gamma_{k+1,r}}^*\cap\Lambda_{\Gamma_{k,r}}^*$ into consideration, we obtain
\begin{eqnarray*}
\sum_{\sigma\in\Lambda_{\Gamma_{k+1,r}}^*}M_{k+1,r}(\sigma)\leq DM^H\sum_{\sigma\in\Lambda_{\Gamma_{k,r}}^*}M_{k,r}(\sigma).
\end{eqnarray*}
Let $d_0:=D\max\{M^H, N^H\}$. Combining the above analysis, we deduce
\begin{eqnarray*}
\phi_{k+1,r}&=&\sum_{h=0}^{l_{1k}-1}\sum_{\sigma\in\Omega_h}M_{k+1,r}(\sigma)+\sum_{h=l_{1,k}}^{l_{1,k+1}-1}\sum_{\sigma\in\Omega_h}M_{k+1,r}(\sigma)\\&&+N_{k+1,r}
+\sum_{\sigma\in\Lambda_{\Gamma_{k+1,r}}^*}M_{k+1,r}(\sigma)\\&\leq& d_0
\sum_{h=0}^{l_{1k}-1}\sum_{\sigma\in\Omega_h}M_{k,r}(\sigma)+d_0\sum_{\sigma\in\Omega_{l_{1,k}-1}}M_{k,r}(\sigma)\\
&&+d_0N_{k,r}+d_0\sum_{\sigma\in\Lambda_{\Gamma_{k,r}}^*}M_{k,r}(\sigma)\\
&\leq&(d_0+1)\phi_{k,r}
\end{eqnarray*}
The lemma follows by setting $d_1:=d_0+1$.
\end{proof}
\begin{lemma}\label{g7}
Let $\mu$ be as stated in Theorem \ref{mthm1}. Then we have
\begin{eqnarray*}
e^r_{\phi_{k,r},r}(\mu)\leq p_0|C|^r\sum_{\sigma\in\Psi_{k,r}}\sum_{\rho\in\Gamma_{k,r}(\sigma)}p_\sigma s_\sigma^r t_\omega c_\omega^r+\sum_{\sigma\in\Gamma_{k,r}}p_\sigma s_\sigma^r.
\end{eqnarray*}
\end{lemma}
\begin{proof}
For each $\sigma\in\Psi_{k,r}$ and $\rho\in\Gamma_{k,r}(\sigma)$, we choose an arbitrary point $a_\rho\in C_\rho=g_\rho(C)$; for every $\sigma\in\Gamma_{k,r}$, let $a_\sigma$ be an arbitrary point of $f_\sigma(K)$. We denote by $\alpha$ the set of all these points. Then ${\rm card}(\alpha)\leq\varphi_{k,r}$. Hence, using Lemma \ref{olsen} (d), we deduce
\begin{eqnarray*}
e^r_{\phi_{k,r},r}(\mu)&\leq&\sum_{\sigma\in\Psi_{k,r}}\sum_{\rho\in\Gamma_{k,r}(\sigma)}\int_{f_\sigma(C_\rho)}d(x,a_\rho)^rd\mu(x)\\&&\;\;\;\;+
\sum_{\sigma\in\Gamma_{k,r}}\int_{f_\sigma(K)}d(x,a_\sigma)^rd\mu(x)\\&\leq&\sum_{\sigma\in\Psi_{k,r}}\sum_{\rho\in\Gamma_{k,r}(\sigma)}
\mu(f_\sigma(C_\rho))|f_\sigma(C_\rho)|^r\\&&\;\;\;\;+\sum_{\sigma\in\Gamma_{k,r}}\mu(f_\sigma(K))|f_\sigma(K)|^r\\&=&\sum_{\sigma\in\Psi_{k,r}}\sum_{\rho\in\Gamma_{k,r}(\sigma)}
p_0p_\sigma t_\rho s_\sigma^rc_\rho^r|C|^r+\sum_{\sigma\in\Gamma_{k,r}}p_\sigma s_\sigma^r.
\end{eqnarray*}
This completes the proof of the lemma.
\end{proof}

In the following, we are devoted to establishing a lower estimate for the quantization error of $\mu$. Compared with the upper one that we have just obtained, we need to make more effort by choosing a suitable subset of $K$ which consists of pairwise disjoint compact sets, so that the techniques in \cite{Zhu:08} is applicable. We have to take care of the OSC both for $(f_i)_{i=1}^N$ and for $(g_i)_{i=1}^M$.
\begin{lemma}
There exists a bounded open set $W$ such that $C\subset W$ and
\begin{eqnarray}\label{s23}
g_i(W)\subset W,\;\;{\rm cl}(W)\cap f_i({\rm cl}(U))=\emptyset\;\; {\rm for\; all} \;\;1\leq i\leq N.
\end{eqnarray}
\end{lemma}
\begin{proof}
By the assumption (A4), $C\cap f_i(\rm cl (U))=\emptyset$ for every $1\leq i\leq N$. Set
\begin{equation*}
\epsilon_0:=\min_{1\leq i\leq N}d(C, f_i(\rm cl (U))).
\end{equation*}
Let $B(x,\epsilon)$ denote the open ball of $\epsilon$ which is centered at $x$. We define
\begin{equation*}
W:=\bigcup_{x\in C}B(x,\epsilon_0/2).
\end{equation*}
Then $W$ is an open set. One can see that $W\cap f_i(\rm cl(U))=\emptyset$ for every $1\leq i\leq N$. Note that $C=\bigcup_{i=1}^M g_i(C)$. For every $x\in C$, we have $g_i(x)\in g_i(C)\subset C$. Hence,
\begin{eqnarray*}
g_i(B(x,\epsilon_0/2))=B(g_i(x),c_i\epsilon_0/2)\subset B(g_i(x),\epsilon_0/2)\subset W.
\end{eqnarray*}
This implies that $g_i(W)\subset W$. Thus, the open set $W$ satisfies (\ref{s23}).
\end{proof}

Let $A^c$ denote the complement of a set $A\subset\mathbb{R}^q$. We set $\delta_0:=d(C,W^c)$.
Let $U$ be the open set from the IOSC. Let $J$ be the same as in \cite[Theorem 3.2]{Gr:95} (cf. \cite{Schief:94}), i.e.,
$J$ is a nonempty compact set satisfying
\begin{eqnarray*}
J={\rm cl}({\rm int}(J)); \;\;{\rm int}(J)\cap C\neq\emptyset;\;\;
g_i(J)\subset J,\;1\leq j\leq M; \\g_i({\rm int}(J))\cap g_j({\rm int}(J))=\emptyset,\;\;1\leq i\neq j\leq M.
\end{eqnarray*}
\begin{lemma}\label{g9}
Let $V:={\rm int}(J)\cap U\cap W$. There exists a $\tau^{(0)}\in\Omega^*\setminus\{\theta\}$ such that $f_{\tau^{(0)}}(K)\subset U$; there exists a $\rho^{(0)}\in\Phi^*\setminus\{\theta\}$ such that $C_{\rho^{(0)}}\subset V$.
\end{lemma}
\begin{proof}
By (A3), we have, $K\cap U\neq\emptyset$. So we may choose a word $\tau^{(0)}\in\Omega^*$ such that $f_{\tau^{(0)}}(K)\subset U$.
By Lemma \ref{olsen} (c), we have, $\mu(U^c)=0$. Thus, by (\ref{attractingmeasure}), we know that $\nu(U^c)=0$, i.e., $\nu(U)=1$. According to \cite[Lemma 3.3]{Gr:95}, we have, $\nu(\rm int(J))=1$. Note that $\nu(W)=\nu(C)=1$. Hence,
\begin{equation*}
\nu(V)=\nu({\rm int}(J)\cap U\cap W)=1.
 \end{equation*}
As a result, we have that $V\cap C\neq\emptyset$. Thus, we may choose a non-empty word $\rho^{(0)}\in\Phi^*$ such that $C_{\rho^{(0)}}\subset V$.  We set
\begin{equation}\label{s24}
\delta_1:=d(f_{\tau^{(0)}}(K),U^c).,\;\;\delta_2:=d(C_{\rho^{(0)}},V^c).
\end{equation}
\end{proof}

For each $k\geq 1$, let $\Gamma_{k,r}$ and $\Psi_{k,r}$ be as defined in (\ref{s7}) and (\ref{s8}). Let $\tau^{(0)}$ and $\rho^{(0)}$ be as chosen above.
We define
\begin{eqnarray*}\label{s10}
G_{k,r}:=\bigg(\bigcup_{\sigma\in\Psi_{k,r}}\bigcup_{\rho\in\Gamma_{k,r}(\sigma)}f_{\sigma\ast\tau^{(0)}}(C_{\rho\ast\rho^{(0)}})\bigg)
\cup\bigg(\bigcup_{\sigma\in\Gamma_{k,r}}f_{\sigma\ast\tau^{(0)}}(K)\bigg).
\end{eqnarray*}
For every $x\in G_{k,r}$ and $\epsilon>0$, by Lemma \ref{g1} one can see that $\mu(B(x,\epsilon))>0$. Hence, we have $K\supset G_{k,r}$ for every $k\geq 1$. We will use $G_{k,r}$ to obtain a lower bound for the quantization error of $\mu$.
For convenience, we write
\begin{eqnarray*}
\mathcal{F}_{k,r}:=\big\{f_{\sigma\ast\tau^{(0)}}(C_{\rho\ast\rho^{(0)}}):\rho\in\Gamma_{k,r}(\sigma),\sigma\in\Psi_{k,r}\big\}
\cup\big\{f_{\sigma\ast\tau^{(0)}}(K)\big\}_{\sigma\in\Gamma_{k,r}}.
\end{eqnarray*}

Let $\delta_1,\delta_2$ be as defined in (\ref{s24}) and let $\tau^{(0)}$ be the same as in Lemma \ref{g9}. Define
\begin{equation}\label{s34}
\delta_3:=d(C,U^c)\;({\rm see}\;(A3));\;\delta:=\min\{\delta_0,\delta_1,\delta_2,\delta_3\}.
\end{equation}

Next, we give an estimate of the distance between every pair $A_1,A_2$ of sets in $\mathcal{F}_{k,r}$. We will show that, For every pair $A_1,A_2\in\mathcal{F}_{k,r}$, we have
\begin{equation}\label{s12}
d(A_1,A_2)\geq\delta\max\{|A_1|,|A_2|\}.
\end{equation}
For the reader's convenience, we divide the proof of (\ref{s12}) into four lemmas according to four distinct cases.
\begin{lemma}
For every $\sigma\in\Psi_{k,r}$ and $\rho,\omega\in\Gamma_{k,r}(\sigma)$ with $\rho\neq\omega$, we have
\begin{eqnarray*}
&&d(f_{\sigma\ast\tau^{(0)}}(C_{\rho\ast\rho^{(0)}}),f_{\sigma\ast\tau^{(0)}}(C_{\omega\ast\rho^{(0)}}))
\\&&\;\;\geq\delta_2\max\{|f_{\sigma\ast\tau^{(0)}}(C_{\omega\ast\rho^{(0)}})|,|f_{\sigma\ast\tau^{(0)}}(C_{\rho\ast\rho^{(0)}})|\}.
\end{eqnarray*}
\end{lemma}
\begin{proof}
Clearly, $\rho,\omega$ are incomparable. Let $h:=\min\{j:\rho_h\neq\omega_h\}$. We write
\begin{equation*}
\omega=\omega|_{h-1}\ast\omega_h\ast\widetilde{\omega},\;\rho=\omega|_{h-1}\ast\rho_h\ast\widetilde{\rho}.
\end{equation*}
By the similarity of $f_{\sigma\ast\tau^{(0)}}$ and $g_\rho,g_\omega$, we deduce
\begin{eqnarray*}
d(f_{\sigma\ast\tau^{(0)}}(C_{\rho\ast\rho^{(0)}}),f_{\sigma\ast\tau^{(0)}}(C_{\omega\ast\rho^{(0)}}))=s_{\sigma\ast\tau^{(0)}} c_{\omega|_{h-1}}d(C_{\omega_h\ast\widetilde{\omega}\ast\rho^{(0)}},C_{\rho_h\ast\widetilde{\rho}\ast\rho^{(0)}}).
\end{eqnarray*}
Note that $C_{\omega_h\ast\widetilde{\omega}\ast\rho^{(0)}}\subset g_{\omega_h\ast\widetilde{\omega}}(V)\subset g_{\omega_h\ast\widetilde{\omega}}(\rm int(J))$ and $C_{\rho_h\ast\widetilde{\rho}\ast\rho^{(0)}}\subset g_{\rho_h\ast\widetilde{\rho}}(V)\subset g_{\rho_h\ast\widetilde{\rho}}(\rm int(J))$. By the OSC for $(g_i)_{i=1}^M$, we have
\begin{eqnarray*}
g_{\omega_h\ast\widetilde{\omega}}(\rm int (J))\cap g_{\rho_h\ast\widetilde{\rho}}(\rm int(J))\subset g_{\omega_h}(\rm int(J))\cap g_{\rho_h}(\rm int(J))=\emptyset.
\end{eqnarray*}
Using the above facts, we further deduce
\begin{eqnarray*}
&&d(f_{\sigma\ast\tau^{(0)}}(C_{\rho\ast\rho^{(0)}}),f_{\sigma\ast\tau^{(0)}}(C_{\omega\ast\rho^{(0)}}))\\&&\geq s_{\sigma\ast\tau^{(0)}} c_{\omega|_{h-1}}\max\{d(C_{\omega_h\ast\widetilde{\omega}\ast\rho^{(0)}}, g_{\omega_h\ast\widetilde{\omega}}(V)^c),d(C_{\rho_h\ast\widetilde{\rho}\ast\rho^{(0)}}, g_{\rho_h\ast\widetilde{\rho}}(V)^c)\}\\&&=s_{\sigma\ast\tau^{(0)}} c_{\omega|_{h-1}}\max\{d(C_{\omega_h\ast\widetilde{\omega}\ast\rho^{(0)}}, g_{\omega_h\ast\widetilde{\omega}}(V^c)),d(C_{\rho_h\ast\widetilde{\rho}\ast\rho^{(0)}}, g_{\rho_h\ast\widetilde{\rho}}(V^c))\}\\
&&\geq s_{\sigma\ast\tau^{(0)}} c_{\omega|_{h-1}}\max\{c_{\omega_h\ast\widetilde{\omega}},c_{\rho_h\ast\widetilde{\rho}}\}\delta_2\\&&
=s_{\sigma\ast\tau^{(0)}} \max\{c_\omega,c_\rho\}\delta_2\\
&&\geq\delta_2\max\{|f_{\sigma\ast\tau^{(0)}}(C_{\omega\ast\rho^{(0)}})|,|f_{\sigma\ast\tau^{(0)}}(C_{\rho\ast\rho^{(0)}})|\}.
\end{eqnarray*}
Here we have used the fact that $|C|\leq|K|=1$. The lemma follows.
\end{proof}
\begin{lemma}
Let $\sigma,\tau\in\Psi_{k,r},\sigma\neq\tau$ and $\rho\in\Gamma_{k,r}(\sigma),\omega\in\Gamma_{k,r}(\tau)$. We have
\begin{eqnarray*}
&&d(f_{\sigma\ast\tau^{(0)}}(C_{\rho\ast\rho^{(0)}}),f_{\tau\ast\tau^{(0)}}(C_{\omega\ast\rho^{(0)}}))\\&&\geq \delta\max\{|f_{\sigma\ast\tau^{(0)}}(C_{\rho\ast\rho^{(0)}})|,|f_{\tau\ast\tau^{(0)}}(C_{\omega\ast\rho^{(0)}}))|\}.
\end{eqnarray*}
\end{lemma}
\begin{proof}
We have the following two cases.

Case 1: $|\sigma|=|\tau|$. In this case, $\sigma,\tau$ are clearly incomparable since $\sigma\neq\tau$. Note that $C_{\rho\ast\rho^{(0)}},C_{\omega\ast\rho^{(0)}}\subset C\subset K$. By the IOSC, $f_\sigma(U)\cap f_\tau(U)=\emptyset$. Thus,
\begin{eqnarray*}
f_{\sigma\ast\tau^{(0)}}(C_{\rho\ast\rho^{(0)}})\subset f_\sigma(U),\;f_{\tau\ast\tau^{(0)}}(C_{\omega\ast\rho^{(0)}})\subset f_\tau(U).
\end{eqnarray*}
Using these facts, we deduce
\begin{eqnarray*}
&&d(f_{\sigma\ast\tau^{(0)}}(C_{\rho\ast\rho^{(0)}}),f_{\tau\ast\tau^{(0)}}(C_{\omega\ast\rho^{(0)}}))\\&&\geq
\max\{d(f_{\sigma\ast\tau^{(0)}}(C_{\rho\ast\rho^{(0)}}),f_\sigma(U^c)),d(f_{\tau\ast\tau^{(0)}}(C_{\omega\ast\rho^{(0)}}),f_\tau(U^c)))\}\\&&\geq
\max\{d(f_{\sigma\ast\tau^{(0)}}(K),f_\sigma(U^c)),d(f_{\tau\ast\tau^{(0)}}(K),f_\tau(U^c)))\}\\
&&\geq \max\{s_\sigma\delta_1,s_\tau\delta_1\}\\&&
\geq\delta_1\max\{|f_{\sigma\ast\tau^{(0)}}(C_{\rho\ast\rho^{(0)}})|,|f_{\sigma\ast\tau^{(0)}}(C_{\omega\ast\rho^{(0)}}))|\}
\end{eqnarray*}

Case 2:  $|\sigma|\neq|\tau|$. Without loss of generality, we assume that $|\sigma|<|\tau|$. We again distinguish have two subcases:

Case 2a $\sigma\nprec\tau$. In this case, we set $h:=\min\{i:\sigma_i\neq\tau_i\}$ and write
\begin{equation}\label{s9}
\sigma=\sigma|_{h-1}\ast\sigma_h\ast\widetilde{\sigma},\;\tau=\sigma|_{h-1}\ast\tau_h\ast\widetilde{\tau}.
\end{equation}
Note that $\sigma_h\neq\tau_h$. By the IOSC, we have, $f_{\sigma_h}(U)\cap f_{\tau_h}(U)=\emptyset$ and
\begin{eqnarray*}
f_{\sigma_h\ast\widetilde{\sigma}\ast\tau^{(0)}}(C_{\rho\ast\rho^{(0)}})\subset
f_{\sigma_h\ast\widetilde{\sigma}}(f_{\tau^{(0)}}(C))\subset f_{\sigma_h\ast\widetilde{\sigma}}(U)\subset f_{\sigma_h}(U),\\
f_{\tau_h\ast\widetilde{\sigma}\ast\tau^{(0)}}(C_{\omega\ast\rho^{(0)}})\subset
f_{\tau_h\ast\widetilde{\sigma}}(f_{\tau^{(0)}}(C))\subset f_{\tau_h\ast\widetilde{\tau}}(U)\subset f_{\tau_h}(U).
\end{eqnarray*}
Thus, using the similarity of the mappings, we deduce
\begin{eqnarray*}
&&d(f_{\sigma\ast\tau^{(0)}}(C_{\rho\ast\rho^{(0)}}),f_{\tau\ast\tau^{(0)}}(C_{\omega\ast\rho^{(0)}}))\\&&\geq s_{\sigma|_{h-1}}d(f_{\sigma_h\ast\widetilde{\sigma}\ast\tau^{(0)}}(C_{\rho\ast\rho^{(0)}}), f_{\tau_h\ast\widetilde{\sigma}\ast\tau^{(0)}}(C_{\omega\ast\rho^{(0)}})\\&&\geq s_{\sigma|_{h-1}}\max\{d(f_{\sigma_h\ast\widetilde{\sigma}\ast\tau^{(0)}}(C_{\rho\ast\rho^{(0)}}), f_{\sigma_h\ast\widetilde{\sigma}}(U^c)),\\&&\;\;\;\;d(f_{\tau_h\ast\widetilde{\tau}\ast\tau^{(0)}}(C_{\omega\ast\rho^{(0)}}), f_{\tau_h\ast\widetilde{\tau}}(U^c)))\}\\&&\geq s_{\sigma|_{h-1}}\max\{d(f_{\sigma_h\ast\widetilde{\sigma}\ast\tau^{(0)}}(K), f_{\sigma_h\ast\widetilde{\sigma}}(U^c)),\\&&\;\;\;\;d(f_{\tau_h\ast\widetilde{\tau}\ast\tau^{(0)}}(K), f_{\tau_h\ast\widetilde{\tau}}(U^c)))\}\\
&&\geq s_{\sigma|_{h-1}}\max\{s_{\sigma_h\ast\widetilde{\sigma}},s_{\tau_h\ast\widetilde{\tau}}\}\delta_1
=\max\{s_\sigma,s_\tau\}\delta_1\\
&&\geq\delta_1\max\{|f_{\sigma\ast\tau^{(0)}}(C_{\rho\ast\rho^{(0)}})|,|f_{\tau\ast\tau^{(0)}}(C_{\omega\ast\rho^{(0)}})|\}.
\end{eqnarray*}

Case 2b: $\sigma\prec\tau$. In this case, we write $\tau=\sigma\ast\widetilde{\tau}$. Then
\begin{eqnarray}\label{s21}
&&d(f_{\sigma\ast\tau^{(0)}}(C_{\rho\ast\rho^{(0)}}),f_{\tau\ast\tau^{(0)}}(C_{\omega\ast\rho^{(0)}}))\nonumber
\\&&=d(f_{\sigma\ast\tau^{(0)}}(C_{\rho\ast\rho^{(0)}}),f_{\sigma\ast\widetilde{\tau}\ast\tau^{(0)}}(C_{\omega\ast\rho^{(0)}}))\nonumber
\\&&=s_\sigma d(f_{\tau^{(0)}}(C_{\rho\ast\rho^{(0)}}),f_{\widetilde{\tau}\ast\tau^{(0)}}(C_{\omega\ast\rho^{(0)}}))
\end{eqnarray}
Again, we need to distinguish two subcases.

Case 2b (i): $\tau^{(0)}\nprec\widetilde{\tau}\ast\tau^{(0)}=:\hat{\tau}$. We write $h:=\min\{i:\tau^{(0)}_i\neq\hat{\tau}_i\}$ and
\begin{eqnarray*}
\tau^{(0)}=\tau^{(0)}|_{h-1}\ast\tau^{(0)}_h\ast\widetilde{\tau^{(0)}},\;\;
\widetilde{\tau}\ast\tau^{(0)}=\tau^{(0)}|_{h-1}\ast\hat{\tau}_h\ast\widetilde{\hat{\tau}}.
\end{eqnarray*}
Note that $C_{\rho\ast\rho^{(0)}}=g_\rho(C_{\rho^{(0)}})\subset g_\rho(C)\subset C\subset U$ and similarly $C_{\omega\ast\rho^{(0)}}\subset U$. Hence,
\begin{eqnarray}
f_{\tau^{(0)}_h\ast\widetilde{\tau^{(0)}}}(C_{\rho\ast\rho^{(0)}})\subset f_{\tau^{(0)}_h\ast\widetilde{\tau^{(0)}}}(U)\subset f_{\tau^{(0)}_h}(U),\label{s28}\\f_{\hat{\tau}_h\ast\widetilde{\hat{\tau}}}(C_{\omega\ast\rho^{(0)}})\subset f_{\hat{\tau}_h\ast\widetilde{\hat{\tau}}}(U)\subset f_{\hat{\tau}_h}(U).\nonumber
\end{eqnarray}
Since $f_{\tau^{(0)}_h}(U)\cap f_{\hat{\tau}_h}(U)=\emptyset$, we have that $f_{\tau^{(0)}_h\ast\widetilde{\tau^{(0)}}}(U)\cap f_{\hat{\tau}_h\ast\widetilde{\hat{\tau}}}(U)=\emptyset$.
Hence, by (\ref{s21}),
\begin{eqnarray*}
&&d(f_{\sigma\ast\tau^{(0)}}(C_{\rho\ast\rho^{(0)}}),f_{\tau\ast\tau^{(0)}}(C_{\omega\ast\rho^{(0)}}))
\\&&=s_\sigma d(f_{\tau^{(0)}}(C_{\rho\ast\rho^{(0)}}),f_{\widetilde{\tau}\ast\tau^{(0)}}(C_{\omega\ast\rho^{(0)}}))\\&&
=s_\sigma s_{\tau^{(0)}|_{h-1}}d(f_{\tau^{(0)}_h\ast\widetilde{\tau^{(0)}}}(C_{\rho\ast\rho^{(0)}}),
f_{\hat{\tau}_h\ast\widetilde{\hat{\tau}}}(C_{\omega\ast\rho^{(0)}}))\\&&
\geq s_\sigma s_{\tau^{(0)}|_{h-1}}\max\big\{d(f_{\tau^{(0)}_h\ast\widetilde{\tau^{(0)}}}(C_{\rho\ast\rho^{(0)}}),f_{\tau^{(0)}_h\ast\widetilde{\tau^{(0)}}}(U^c)),\\&&
\;\;\;\;\;\;\;
d(f_{\hat{\tau}_h\ast\widetilde{\hat{\tau}}}(C_{\omega\ast\rho^{(0)}}),f_{\hat{\tau}_h\ast\widetilde{\hat{\tau}}}(U^c))\big\}\\&&\geq
s_\sigma s_{\tau^{(0)}|_{h-1}}\max\{s_{\tau^{(0)}_h\ast\widetilde{\tau^{(0)}}},s_{\hat{\tau}_h\ast\widetilde{\hat{\tau}}}\}\delta_3\\&&=
\max\{s_{\sigma\ast\tau^{(0)}},s_{\tau\ast\tau^{(0)}}\}\delta_3\\&&
\geq\delta_3\max\{|f_{\sigma\ast\tau^{(0)}}(C_{\rho\ast\rho^{(0)}})|,|f_{\tau\ast\tau^{(0)}}(C_{\omega\ast\rho^{(0)}})|\}.
\end{eqnarray*}

Case 2b (ii): $\tau^{(0)}\prec\hat{\tau}=\widetilde{\tau}\ast\tau^{(0)}$. We write $\hat{\tau}=\tau^{(0)}\ast\overline{\hat{\tau}}$. Then by (\ref{s21}),
\begin{eqnarray}\label{s22}
&&d(f_{\sigma\ast\tau^{(0)}}(C_{\rho\ast\rho^{(0)}}),f_{\tau\ast\tau^{(0)}}(C_{\omega\ast\rho^{(0)}}))
\nonumber\\&&=s_\sigma d(f_{\tau^{(0)}}(C_{\rho\ast\rho^{(0)}}),f_{\widetilde{\tau}\ast\tau^{(0)}}(C_{\omega\ast\rho^{(0)}}))\nonumber\\&&
=s_\sigma d(f_{\tau^{(0)}}(C_{\rho\ast\rho^{(0)}}),f_{\tau^{(0)}\ast\overline{\hat{\tau}}}(C_{\omega\ast\rho^{(0)}}))\nonumber\\&&
=s_\sigma s_{\tau^{(0)}}d(C_{\rho\ast\rho^{(0)}}),f_{\overline{\hat{\tau}}}(C_{\omega\ast\rho^{(0)}})).
\end{eqnarray}
Note that $C_{\rho\ast\rho^{(0)}}\subset g_\rho(W)\subset W$, and $f_{\overline{\hat{\tau}}}(C_{\omega\ast\rho^{(0)}})\subset f_{\overline{\hat{\tau}}}(C))\subset f_{\overline{\hat{\tau}}}(U))$ by (A4). Hence,
\begin{eqnarray*}
W\cap f_{\overline{\hat{\tau}}}(C_{\omega\ast\rho^{(0)}})&\subset& W\cap f_{\overline{\hat{\tau}}}(C)
\subset W\cap f_{\overline{\hat{\tau}}}(U)\subset  W\cap f_{\overline{\hat{\tau}}_1}(U)=\emptyset.
\end{eqnarray*}
Using this and (\ref{s22}), we deduce
\begin{eqnarray*}
&&d(f_{\sigma\ast\tau^{(0)}}(C_{\rho\ast\rho^{(0)}}),f_{\tau\ast\tau^{(0)}}(C_{\omega\ast\rho^{(0)}}))
\\&&=s_\sigma s_{\tau^{(0)}}d(C_{\rho\ast\rho^{(0)}}),f_{\overline{\hat{\tau}}}(C_{\omega\ast\rho^{(0)}}))
\\&&\geq s_\sigma s_{\tau^{(0)}}d(C_{\rho\ast\rho^{(0)}},W^c)\geq s_\sigma s_{\tau^{(0)}}d(C,W^c)\\&&\geq  s_\sigma s_{\tau^{(0)}}\delta_0\geq\delta_0\max\{|f_{\sigma\ast\tau^{(0)}}(C_{\rho\ast\rho^{(0)}})|,|f_{\tau\ast\tau^{(0)}}(C_{\omega\ast\rho^{(0)}})|\}.
\end{eqnarray*}
For the last inequality in the preceding display, we have used the fact that $|\sigma|<|\tau|$.
The lemma follows by combining the analysis of the above two cases.
\end{proof}
\begin{lemma}
Let $\sigma\in\Psi_{k,r},\rho\in\Gamma_{k,r}(\sigma)$ and $\tau\in\Gamma_{k,r}$. Then
\begin{eqnarray}\label{s27}
&&d(f_{\sigma\ast\tau^{(0)}}(C_{\rho\ast\rho^{(0)}}),f_{\tau\ast\tau^{(0)}}(K))\nonumber
\\&&\;\;\geq\delta\max\{|f_{\sigma\ast\tau^{(0)}}(C_{\rho\ast\rho^{(0)}})|,|f_{\tau\ast\tau^{(0)}}(K)|\}.
\end{eqnarray}
\end{lemma}
\begin{proof}
First, we note the following two facts:

(F1) For every $\sigma\in\Lambda_{\Gamma_{k,r}}^*$, by the definition, there exists a $\tau\in\Gamma_{k,r}$ such that $|\tau|>|\sigma|$ and $\sigma\prec\tau$. Since $\Gamma_{k,r}$ is a finite maximal antichain, we know that the predecessors of $\sigma$ do not belong to $\Gamma_{k,r}$. Thus, for $\sigma\in\Lambda_{\Gamma_{k,r}}^*$ and $\tau\in\Gamma_{k,r}$, we have, either $\sigma\prec\tau,|\sigma|<|\tau|$, or, $\sigma,\tau$ are incomparable.

(F2) For $\sigma\in\bigcup_{h=0}^{l_{1,k}-1}\Omega_h$ and $\tau\in\Gamma_{k,r}$, we have, $|\sigma|\leq l_{1,k}-1<|\tau|$. We have only two possible cases: either $\sigma\prec\tau$, or $\sigma\nprec\tau$.

Combining the above analysis, we conclude that, for every $\sigma\in\Psi_{k,r}$ and $\tau\in\Gamma_{k,r}$, we have, either $\sigma\prec\tau,|\sigma|<|\tau|$, or, $\sigma,\tau$ are incomparable.

Next, we complete the proof analogously to Case 2 of the preceding lemma. We distinguish two cases.

Case I: $\sigma,\tau$ are incomparable. We write $\sigma,\tau$ as in (\ref{s9}), namely,
\begin{equation*}
\sigma=\sigma|_{h-1}\ast\sigma_h\ast\widetilde{\sigma},\;\tau=\sigma|_{h-1}\ast\tau_h\ast\widetilde{\tau}.
\end{equation*}
One can replace $C_{\omega\ast\rho^{(0)}}$ with $K$ in the proof of Case 2a and get (\ref{s27}) conveniently.

Case II: $\sigma\prec\tau$ and $|\sigma|<|\tau|$. In this case, we write $\tau=\sigma\ast\widetilde{\tau}$. Then
\begin{eqnarray}\label{s26}
&&d(f_{\sigma\ast\tau^{(0)}}(C_{\rho\ast\rho^{(0)}}),f_{\tau\ast\tau^{(0)}}(K))\nonumber
\\&&=d(f_{\sigma\ast\tau^{(0)}}(C_{\rho\ast\rho^{(0)}}),f_{\sigma\ast\widetilde{\tau}\ast\tau^{(0)}}(K))\nonumber
\\&&=s_\sigma d(f_{\tau^{(0)}}(C_{\rho\ast\rho^{(0)}}),f_{\widetilde{\tau}\ast\tau^{(0)}}(K))
\end{eqnarray}
As in the proof of the preceding lemma, we need to distinguish two subcases.

Case II (i): $\tau^{(0)}\nprec\widetilde{\tau}\ast\tau^{(0)}=:\hat{\tau}$. We write $h:=\min\{i:\tau^{(0)}_i\neq\hat{\tau}_i\}$ and
\begin{eqnarray*}
\tau^{(0)}=\tau^{(0)}|_{h-1}\ast\tau^{(0)}_h\ast\widetilde{\tau^{(0)}},\;\;
\widetilde{\tau}\ast\tau^{(0)}=\tau^{(0)}|_{h-1}\ast\hat{\tau}_h\ast\widetilde{\hat{\tau}}.
\end{eqnarray*}
As in Case 2b(i), (\ref{s28}) holds. By Lemma \ref{olsen} (d), $K\subset{\rm cl}(U)$. Hence, we deduce
\begin{eqnarray*}
f_{\hat{\tau}_h\ast\widetilde{\hat{\tau}}}(K)\subset f_{\hat{\tau}_h\ast\widetilde{\hat{\tau}}}(\rm cl(U))\subset f_{\hat{\tau}_h}(\rm cl(U)).
\end{eqnarray*}
Since $f_{\tau^{(0)}_h}(U)\cap f_{\hat{\tau}_h}(U)=\emptyset$, we have, $f_{\tau^{(0)}_h\ast\widetilde{\tau^{(0)}}}(U)\cap f_{\hat{\tau}_h\ast\widetilde{\hat{\tau}}}(\rm cl(U))=\emptyset$. On the other hand, we have $C\subset U$, so $f_{\tau^{(0)}_h\ast\widetilde{\tau^{(0)}}}(C)\subset f_{\tau^{(0)}_h\ast\widetilde{\tau^{(0)}}}(U)$.
Hence, by (\ref{s26}),
\begin{eqnarray*}
&&d(f_{\sigma\ast\tau^{(0)}}(C_{\rho\ast\rho^{(0)}}),f_{\tau\ast\tau^{(0)}}(K))
\\&&=s_\sigma d(f_{\tau^{(0)}}(C_{\rho\ast\rho^{(0)}}),f_{\widetilde{\tau}\ast\tau^{(0)}}(K))\\&&
=s_\sigma s_{\tau^{(0)}|_{h-1}}d(f_{\tau^{(0)}_h\ast\widetilde{\tau^{(0)}}}(C_{\rho\ast\rho^{(0)}}),
f_{\hat{\tau}_h\ast\widetilde{\hat{\tau}}}(K))\\&&
\geq s_\sigma s_{\tau^{(0)}|_{h-1}}d(f_{\tau^{(0)}_h\ast\widetilde{\tau^{(0)}}}(C_{\rho\ast\rho^{(0)}}),f_{\tau^{(0)}_h\ast\widetilde{\tau^{(0)}}}(U^c))\\&&
\geq s_\sigma s_{\tau^{(0)}|_{h-1}}d(f_{\tau^{(0)}_h\ast\widetilde{\tau^{(0)}}}(C),f_{\tau^{(0)}_h\ast\widetilde{\tau^{(0)}}}(U^c))\\&&\geq
s_\sigma s_{\tau^{(0)}|_{h-1}}s_{\tau^{(0)}_h\ast\widetilde{\tau^{(0)}}}\delta_3=s_\sigma s_{\tau^{(0)}}\delta_3
\\&&\geq\delta_3\max\{|f_{\sigma\ast\tau^{(0)}}(C_{\rho\ast\rho^{(0)}})|,|f_{\tau\ast\tau^{(0)}}(K)|\}.
\end{eqnarray*}
For the last inequality, we have used the fact that $\sigma\prec\tau$.

Case II (ii): $\tau^{(0)}\prec\hat{\tau}=\widetilde{\tau}\ast\tau^{(0)}$. We write $\hat{\tau}=\tau^{(0)}\ast\overline{\hat{\tau}}$. Then as in (\ref{s21}),
\begin{eqnarray}\label{s29}
&&d(f_{\sigma\ast\tau^{(0)}}(C_{\rho\ast\rho^{(0)}}),f_{\tau\ast\tau^{(0)}}(K))\nonumber
\\&&=s_\sigma d(f_{\tau^{(0)}}(C_{\rho\ast\rho^{(0)}}),f_{\widetilde{\tau}\ast\tau^{(0)}}(K))\nonumber\\&&
=s_\sigma d(f_{\tau^{(0)}}(C_{\rho\ast\rho^{(0)}}),f_{\tau^{(0)}\ast\overline{\hat{\tau}}}(K))\nonumber\\&&
=s_\sigma s_{\tau^{(0)}}d(C_{\rho\ast\rho^{(0)}},f_{\overline{\hat{\tau}}}(K)).
\end{eqnarray}
Note that $C_{\rho\ast\rho^{(0)}}\subset g_\rho(W)\subset W$; and by Lemma \ref{olsen} (a), $f_{\overline{\hat{\tau}}}(K)\subset f_{\overline{\hat{\tau}}}(\rm cl(U)))$. Hence,
\begin{eqnarray*}
W\cap f_{\overline{\hat{\tau}}}(K)&\subset& W\cap f_{\overline{\hat{\tau}}}(\rm cl(U))
\subset  W\cap f_{\overline{\hat{\tau}}_1}(\rm cl(U))=\emptyset.
\end{eqnarray*}
Using this and (\ref{s29}), we deduce
\begin{eqnarray*}
&&d(f_{\sigma\ast\tau^{(0)}}(C_{\rho\ast\rho^{(0)}},f_{\tau\ast\tau^{(0)}}(K))
\\&&=s_\sigma s_{\tau^{(0)}}d(C_{\rho\ast\rho^{(0)}},f_{\overline{\hat{\tau}}}(K))
\geq s_\sigma s_{\tau^{(0)}}d(C_{\rho\ast\rho^{(0)}},W^c)\\&&\geq s_\sigma s_{\tau^{(0)}}d(C,W^c)\geq s_\sigma s_{\tau^{(0)}}\delta_0\\&&\geq \delta_0\max\{|f_{\sigma\ast\tau^{(0)}}(C_{\rho\ast\rho^{(0)}})|,|f_{\tau\ast\tau^{(0)}}(K)|\}.
\end{eqnarray*}
As in Case II (i), for the last inequality, we have used the fact that $\sigma\prec\tau$.
By combining the above analysis, the lemma follows.
\end{proof}

\begin{lemma}
For $\sigma,\tau\in\Gamma_{k,r}$, we have
\begin{eqnarray*}
d(f_{\sigma\ast\tau^{(0)}}(K),f_{\tau\ast\tau^{(0)}}(K))
\geq\delta_1\max\{|f_{\sigma\ast\tau^{(0)}}(K)|,|f_{\tau\ast\tau^{(0)}}(K)|\}.
\end{eqnarray*}
\end{lemma}
\begin{proof}
Clearly $\sigma,\tau$ are incomparable. Using (\ref{s9}), we deduce
 \begin{eqnarray*}
&&d(f_{\sigma\ast\tau^{(0)}}(K),f_{\tau\ast\tau^{(0)}}(K))\\
&&\geq s_{\sigma|_h}\max\{d(f_{\sigma_h\ast\widetilde{\sigma}\ast\tau^{(0)}}(K), f_{\sigma_h\ast\widetilde{\sigma}}(U^c)),\\&&\;\;\;\;\;\;d(f_{\tau_h\ast\widetilde{\tau}\ast\tau^{(0)}}(K), f_{\tau_h\ast\widetilde{\tau}}(U^c)))\}\\
&&\geq s_{\sigma|_h}\max\{s_{\sigma_h\ast\widetilde{\sigma}}, s_{\tau_h\ast\widetilde{\tau}})\}\delta_1=\max\{s_\sigma,s_\tau\}\delta_1\\
&&\geq\delta_1\max\{|f_{\sigma\ast\tau^{(0)}}(C_{\rho\ast\rho^{(0)}})|,|f_{\tau\ast\tau^{(0)}}(K)|\}.
\end{eqnarray*}
This completes the proof of the lemma.
\end{proof}

Next, we need to estimate the "energy"----$\mathcal{E}(A):=\mu(A)|A|^r$ of the sets $A\in\mathcal{F}_{k,r}$. As we will show in the following lemma, the energy of these sets are uniformly comparable. That is,
\begin{lemma}
There exist constants $d_2,d_3>0$ such that
\begin{eqnarray}\label{s13}
d_3\underline{\eta}_r^k\leq\mathcal{E}(A)<d_2\underline{\eta}_r^k,\;\;{\rm for\;all}\;\;A\in\mathcal{F}_{k,r}\;\;{\rm and}\;\;k\geq 1.
\end{eqnarray}
\end{lemma}
\begin{proof}
For each $\sigma\in\Psi_{k,r}$ and $\rho\in\Gamma_{k,r}(\sigma)$, by (\ref{s11}), we have
\begin{eqnarray*}
&&\mathcal{E}(f_{\sigma\ast\tau^{(0)}}(C_{\rho\ast\rho^{(0)}})))\\&&=\mu(f_{\sigma\ast\tau^{(0)}}(C_{\rho\ast\rho^{(0)}}))|f_{\sigma\ast\tau^{(0)}}(C_{\rho\ast\rho^{(0)}})|^r
\\&&=p_0p_{\sigma\ast\tau^{(0)}}t_{\rho\ast\rho^{(0)}}s_{\sigma\ast\tau^{(0)}}^r
c_{\rho\ast\rho^{(0)}}^r|C|^r\\&&=p_0p_{\tau^{(0)}}t_{\rho^{(0)}}s_{\tau^{(0)}}^rc_{\rho^{(0)}}^r|C|^rp_\sigma s_\sigma^r t_\rho c_\rho^r\\&&=H_1p_\sigma s_\sigma^r t_\rho c_\rho^r\left\{\begin{array}{ll}\leq H_1\underline{\eta}_r^k\\
\geq H_1H_2\underline{\eta}_r^k\end{array}\right.,
\end{eqnarray*}
where $H_1:=p_0p_{\tau^{(0)}}t_{\rho^{(0)}}s_{\tau^{(0)}}^rc_{\rho^{(0)}}^r|C|^r;\;\;H_2:=\min_{1\leq i\leq M}t_ic_i^r$.

Let $H_3:=p_{\tau^{(0)}}s_{\tau^{(0)}}^r$ and $H_4:=\min_{1\leq i\leq N}p_i c_i^r$. For $\sigma\in\Gamma_{k,r}$, by Lemma \ref{olsen} (d),
\begin{eqnarray*}
&&\mathcal{E}(f_{\sigma\ast\tau^{(0)}}(K))=\mu(f_{\sigma\ast\tau^{(0)}}(K))|f_{\sigma\ast\tau^{(0)}}(K)|^r\\&&=p_\sigma p_{\tau^{(0)}}s_\sigma^r s_{\tau^{(0)}}^r=H_3p_\sigma s_\sigma^r\left\{\begin{array}{ll}\leq H_3\underline{\eta}_r^k\\
\geq H_3H_4\underline{\eta}_r^k\end{array}\right..
\end{eqnarray*}
It suffices to set $d_2:=\max\{H_1,H_3\}$ and $d_3:=\min\{H_1H_2,H_3H_4\}$.
\end{proof}

Let $C_{n,r}(\mu)$ \cite{GL:00} denote the collection of sets $\alpha\in\bigcup_{h\geq 1}^n\mathcal{D}_h$ satisfying (\ref{quanerror}). We also call the points in such a set $\alpha$ $n$-optimal points for $\mu$. With the above preparations, we are able to apply the technique in \cite{Zhu:08} to estimate the number of $\phi_{k,r}$-optimal points lying in the pairwise disjoint neighborhoods of sets in $\mathcal{F}_{k,r}$. Let $\delta$ be as defined in (\ref{s34}). we have
\begin{lemma}\label{g8}
There exists a constant $d_4>0$ such that $\mu(G_{k,r})\geq d_4$. Moreover, there is a constant $L$ independent of $k$ such that, for every $\alpha\in C_{\phi_{k,r},r}(\mu(\cdot|G_{k,r}))$,
\begin{eqnarray}\label{s14}
{\rm card}(\alpha\cap (A)_{8^{-1}\delta|A|})\leq L\;\;{\rm for\;all}\;\;A\in\mathcal{F}_{k,r}.
\end{eqnarray}
\end{lemma}
\begin{proof}
By (\ref{s10}) and Lemma \ref{g1}, we have
\begin{eqnarray*}
\mu(G_{k,r})&\geq&\mu\bigg(\bigcup_{h=0}^{l_{1,k}}\bigcup_{\sigma\in\Omega_h}f_{\sigma\ast\tau^{(0)}}(C_{\rho\ast\rho^{(0)}})\bigg)
\\&=&\sum_{h=0}^{l_{1,k}-1}\sum_{\sigma\in\Omega_h}\sum_{\rho\in\Gamma_{k,r}(\sigma)}p_0p_\sigma t_\rho p_{\tau^{(0)}}t_{\rho^{(0)}}\\&=&
p_0p_{\tau^{(0)}}t_{\rho^{(0)}}\sum_{h=0}^{l_{1,k}-1}\sum_{\sigma\in\Omega_h}p_\sigma\sum_{\rho\in\Gamma_{k,r}(\sigma)} t_\rho \\
&=&p_0p_{\tau^{(0)}}t_{\rho^{(0)}}\sum_{h=0}^{l_{1,k}-1}(1-p_0)^h
\\&=& p_{\tau^{(0)}}t_{\rho^{(0)}}(1-(1-p_0)^{l_{1,k}})\geq p_0p_{\tau^{(0)}}t_{\rho^{(0)}}=:d_4.
\end{eqnarray*}
Let $\alpha\in C_{\phi_{k,r},r}(\mu(\cdot|G_{k,r}))$. Using (\ref{s12}), (\ref{s13}) and the techniques in \cite[Lemma 10]{Zhu:08}, one may find some $L\geq 1$ such that (\ref{s14}) holds.
\end{proof}

\begin{lemma}
For $l\geq 1$, there is a number $B_l>0$ such that for every $A\in\mathcal{F}_{k,r}$,
\begin{eqnarray}\label{s17}
\inf_{\beta\in\mathcal{D}_l}\int_Ad(x,\beta)^rd\mu(x)\geq B_l\mathcal{E}(A).
\end{eqnarray}
In particular, for $\beta\in C_{\phi_{k,r},r}(\mu(\cdot|G_{k,r}))$, there is a constant $\widetilde{L}$ such that
\begin{eqnarray}\label{s18}
\min_{A\in \mathcal{F}_{k,r}}\int_Ad(x,\beta)^rd\mu(x)\geq B_{\widetilde{L}}\mathcal{E}(A).
\end{eqnarray}
\end{lemma}
\begin{proof}
Let $l\geq 1$ and $\beta\in\mathcal{D}_l$ (see (\ref{quanerror})) be given. Set $\gamma:=f_{\sigma\ast\tau^{(0)}}^{-1}(\beta)$.
For every $\sigma\in\Psi_{k,r}$ and $\rho\in\Gamma_{k,r}(\sigma)$, using (\ref{s3}), we deduce
\begin{eqnarray*}
&&\int_{f_{\sigma\ast\tau^{(0)}}(C_{\rho\ast\rho^{(0)}})}d(x,\beta)^rd\mu(x)\\&&\geq p_{\sigma\ast\tau^{(0)}}\int_{f_{\sigma\ast\tau^{(0)}}(C_{\rho\ast\rho^{(0)}})}d(x,\beta)^rd\nu\circ f_{\sigma\ast\tau^{(0)}}^{-1}(x)\\
&&\geq p_{\sigma\ast\tau^{(0)}}\int_{C_{\rho\ast\rho^{(0)}}}d(f_{\sigma\ast\tau^{(0)}}(x),\beta)^rd\nu(x)
\\&&=p_{\sigma\ast\tau^{(0)}}s_{\sigma\ast\tau^{(0)}}^r\int_{C_{\rho\ast\rho^{(0)}}}d(x,\gamma)^rd\nu(x)\\
&&=p_{\sigma\ast\tau^{(0)}}s_{\sigma\ast\tau^{(0)}}^rt_{\rho\ast\rho^{(0)}}\int d(x,\gamma)^rd\nu\circ g_{\rho\ast\rho^{(0)}}^{-1}(x)\\
&&=p_{\sigma\ast\tau^{(0)}}s_{\sigma\ast\tau^{(0)}}^rt_{\rho\ast\rho^{(0)}}c_{\rho\ast\rho^{(0)}}^r\int d(x,g_{\rho\ast\rho^{(0)}}^{-1}(\gamma))^rd\nu(x)\\
&&\geq p_{\sigma\ast\tau^{(0)}}s_{\sigma\ast\tau^{(0)}}^rt_{\rho\ast\rho^{(0)}}c_{\rho\ast\rho^{(0)}}^r e_{l,r}^r(\nu)\\&&\geq e_{l,r}^r(\nu)\mathcal{E}(f_{\sigma\ast\tau^{(0)}}(C_{\rho\ast\rho^{(0)}})),
\end{eqnarray*}
where we have used the fact that $|C|\leq |K|=1$.
Now let $\sigma\in\Gamma_{k,r}$ and let $\gamma$ be as defined above. again, by (\ref{s3}), we have
\begin{eqnarray*}
&&\int_{f_{\sigma\ast\tau^{(0)}}(K)}d(x,\beta)^rd\mu(x)\\&& \geq p_{\sigma\ast\tau^{(0)}}\int_{f_{\sigma\ast\tau^{(0)}}(K)}d(x,\beta)^rd\mu\circ f_{\sigma\ast\tau^{(0)}}^{-1}(x)\\
&&\geq p_{\sigma\ast\tau^{(0)}}\int d(f_{\sigma\ast\tau^{(0)}}(x),\beta)^rd\mu(x)\\&&=p_{\sigma\ast\tau^{(0)}}s_{\sigma\ast\tau^{(0)}}^r\int d(x,\gamma)^rd\mu(x)\\
&&\geq p_{\sigma\ast\tau^{(0)}}s_{\sigma\ast\tau^{(0)}}^re_{l,r}^r(\mu)=e_{l,r}^r(\mu)\mathcal{E}(f_{\sigma\ast\tau^{(0)}}(K)).
\end{eqnarray*}
Thus, (\ref{s17}) follows by setting $B_l:=\min\{e_{l,r}^r(\mu),e_{l,r}^r(\nu)\}$.

Now let $\beta\in C_{\phi_{k,r}}(\mu(\cdot|G_{k,r}))$. By Lemma \ref{g8}, ${\rm card}(\beta\cap(A)_{8^{-1}\delta|A|})\leq L$ for every $A\in \mathcal{F}_{k,r}$. By estimating volumes, one can see that, there is a constant $L_1\in\mathbb{N}$, which is independent of $A$ such that $A$ can be covered by $L_1$ closed balls of radii $\delta|A|/16$. We denote by $\gamma_A$ the centers of such $L_1$ balls. Then
\begin{eqnarray*}
\int_A d(x,\beta)^rd\mu(x)\geq\int_A d(x,(\beta\cap(A)_{8^{-1}\delta|A|})\cup\gamma_A)^rd\mu(x)
\end{eqnarray*}
Note that ${\rm card}((\beta\cap(A)_{8^{-1}\delta|A|})\cup\gamma_A)\leq L+L_1=:L_2$. By (\ref{s17}), we obtain
\begin{eqnarray*}
\int_A d(x,\beta)^rd\mu(x)\geq B_{L_2}\mathcal{E}(A).
\end{eqnarray*}
By setting $\widetilde{L}:=L_2$, (\ref{s18}) follows.
\end{proof}

With the help of the preceding two Lemmas, we are able to give a lower bound for the quantization error of $\mu$.

\begin{lemma}\label{g4}
Let $\mu$ be as stated in Theorem \ref{mthm1}. There exists some constant $d_5>0$, which is independent of $k$, such that
\begin{eqnarray*}
e^r_{\phi_{k,r},r}(\mu)\geq d_5\bigg(\sum_{\sigma\in\Psi_{k,r}}\sum_{\rho\in\Gamma_{k,r}(\sigma)}p_\sigma s_\sigma^r t_\omega c_\omega^r+\sum_{\sigma\in\Gamma_{k,r}}p_\sigma s_\sigma^r\bigg).
\end{eqnarray*}
\end{lemma}
\begin{proof}
Let $\alpha\in C_{\phi_{k,r},r}(\mu)$ and $\beta\in C_{\phi_{k,r},r}(\mu(\cdot|G_{k,r}))$. Then we have
\begin{eqnarray*}
&&e_{\phi_{k,r},r}^r(\mu)=\int d(x,\alpha)^rd\mu(x)\geq\int_{G_{k,r}}d(x,\alpha)^rd\mu(x)\\&&\geq\int_{G_{k,r}} d(x,\beta)^rd\mu(x)=B_{\widetilde{L}}\sum_{A\in \mathcal{F}_{k,r}}\mathcal{E}(A)\;\;\;{\rm by}\;\;(\ref{s18})\\
&&\geq B_{\widetilde{L}}\sum_{\sigma\in\Psi_{k,r}}\sum_{\rho\in\Gamma_{k,r}(\sigma)}p_{\sigma \ast\tau^{(0)}}s_{\sigma\ast\tau^{(0)}}^rt_{\rho\ast\rho^{(0)}}c_{\rho\ast\rho^{(0)}}^r\\&&\;\;\;\;\;+B_{\widetilde{L}}\sum_{\sigma\in\Gamma_{k,r}}p_{\sigma \ast\tau^{(0)}}s_{\sigma\ast\tau^{(0)}}^r\\
&&=d_5\sum_{\sigma\in\Psi_{k,r}}\sum_{\rho\in\Gamma_{k,r}(\sigma)}p_\sigma s_\sigma^rt_\rho c_\rho^r+d_5\sum_{\sigma\in\Gamma_{k,r}}p_\sigma s_\sigma^r,
\end{eqnarray*}
where $d_5:=(p_{\tau^{(0)}}s_{\tau^{(0)}}^rt_{\rho^{(0)}}c_{\rho^{(0)}}^r)B_{\widetilde{L}}$.
This completes the proof of the lemma.
\end{proof}

To check the finiteness or positivity of the upper and lower quantization coefficient, we need one more auxiliary lemma. We write
\begin{equation*}
a(s):=\sum_{i=1}^M(t_ic_i^r)^{\frac{s}{s+r}},\;\;b(s):=\sum_{i=1}^N(p_is_i^r)^{\frac{s}{s+r}},\;\;s>0.
\end{equation*}
Then $a(s_r)=b(t_r)=1$. For $s>0$ and $k\geq 1$, we set
\begin{eqnarray*}
I_k(s):=\sum_{\sigma\in\Psi_{k,r}}\sum_{\rho\in\Gamma_{k,r}(\sigma)}(p_\sigma s_\sigma^r t_\rho c_\rho^r)^{\frac{s}{s+r}}+\sum_{\sigma\in\Gamma_{k,r}}(p_\sigma s_\sigma^r)^{\frac{s}{s+r}}.
\end{eqnarray*}
\begin{lemma}\label{g6}
Let $I_k(s)$ be as defined above. We have
\begin{eqnarray*}
I_k(s_r)\left\{\begin{array}{ll}<(1-b(s_r))^{-1}+1,\;\;\;\;\;\;{\rm if}\;\;s_r>t_r\\
\geq l_{1,k}\;\;\;\;\;\;\;\;\;\;\;\;\;\;\;\;\;\;\;\;\;\;\;\;\;\;\;\;\;{\rm if}\;\;s_r=t_r\end{array}\right..
\end{eqnarray*}
\end{lemma}
\begin{proof}
For every $\sigma\in\Psi_{k,r}$, $\Gamma_{k,r}(\sigma)$ is a finite maximal antichain in $\Phi^*$. Note that $a(s_r)=1$. Hence, we have
$\sum_{\rho\in\Gamma_{k,r}(\sigma)}(t_\rho c_\rho^r)^{\frac{s_r}{s_r+r}}=1$. It follows that
\begin{eqnarray*}
\sum_{\rho\in\Gamma_{k,r}(\sigma)}(p_\sigma s_\sigma^r t_\rho c_\rho^r)^{\frac{s_r}{s_r+r}}=
(p_\sigma s_\sigma^r)^{\frac{s_r}{s_r+r}}\sum_{\rho\in\Gamma_{k,r}(\sigma)}(t_\rho c_\rho^r)^{\frac{s_r}{s_r+r}}=
(p_\sigma s_\sigma^r)^{\frac{s_r}{s_r+r}}.
\end{eqnarray*}
Assume that $s_r>t_r$.Then we have that $b(s_r)<1$. Using this, we deduce (cf. Lemma 2.5 of \cite{Zhu:12})
\begin{equation*}
\sum_{\sigma\in\Gamma_{k,r}}(p_\sigma s_\sigma^r)^{\frac{s_r}{s_r+r}}\leq\sum_{\sigma\in\Omega_{l_{1,k}}}(p_\sigma s_\sigma^r)^{\frac{s_r}{s_r+r}}=b(s_r)^{l_{1,k}}<1.
\end{equation*}
On the other hand, one can easily see
\begin{eqnarray*}
\sum_{\sigma\in\Omega_h}(p_\sigma s_\sigma^r)^{\frac{s_r}{s_r+r}}=b(s_r)^h,\;\;\Lambda_{\Gamma_{k,r}}^*\subset\bigcup_{h=l_{1,k}}^{l_{2,k}-1}\Omega_h.
\end{eqnarray*}
Combining the above analysis, we deduce
\begin{eqnarray*}
I_k(s_r)&=&\sum_{h=0}^{l_{1,k}-1}\sum_{\sigma\in\Omega_h}(p_\sigma s_\sigma^r)^{\frac{s_r}{s_r+r}}+\sum_{\sigma\in\Gamma_{k,r}}(p_\sigma s_\sigma^r)^{\frac{s_r}{s_r+r}}\\&&\;\;\;+\sum_{\sigma\in\Lambda_{\Gamma_{k,r}}^*}(p_\sigma s_\sigma^r)^{\frac{s_r}{s_r+r}}\nonumber\\
&\leq&\sum_{h=0}^{l_{1,k}-1}b(s_r)^h+b(s_r)^{l_{1,k}}+\sum_{h=l_{1,k}}^{l_{2,k}-1}b(s_r)^h\nonumber\\
&\leq&\sum_{h=0}^{l_{2,k}}b(s_r)^h+b(s_r)^{l_{1,k}}<\frac{1}{1-b(s_r)}+1.
\end{eqnarray*}
Now we assume that $s_r=t_r$. Then $a(s_r)=b(s_r)=1$. Note that, for every $\sigma\in\Psi_{k,r}$, we have, $\sum_{\rho\in\Gamma_{k,r}(\sigma)}(t_\rho c_\rho^r)^{\frac{s_r}{s_r+r}}=1$. Using this, we deduce
\begin{eqnarray}
I_k(s_r)&\geq&\sum_{h=0}^{l_{1,k}-1}\sum_{\sigma\in\Omega_h}\sum_{\rho\in\Gamma_{k,r}(\sigma)}(p_\sigma s_\sigma^r t_\rho c_\rho^r)^{\frac{s_r}{s_r+r}}\nonumber\\&=&\sum_{h=0}^{l_{1,k}-1}\sum_{\sigma\in\Omega_h}(p_\sigma s_\sigma^r)^{\frac{s_r}{s_r+r}}\sum_{\rho\in\Gamma_{k,r}(\sigma)}(t_\rho c_\rho^r)^{\frac{s_r}{s_r+r}}\nonumber\\&=&\sum_{h=0}^{l_{1,k}-1}\sum_{\sigma\in\Omega_h}(p_\sigma s_\sigma^r)^{\frac{s_r}{s_r+r}}=l_{1,k}.\label{s33}
\end{eqnarray}
This completes the proof of the lemma.
\end{proof}

\emph{Proof of Theorem \ref{mthm1}}

Recall that $\xi_r=\max\{s_r,t_r\}$.
If $\xi_r=s_r$, then $\sum_{\rho\in\Gamma_{k,r}(\theta)}(t_\rho c_\rho^r)^{\frac{\xi_r}{\xi_r+r}}=1$ since $\Gamma_{k,r}(\theta)$ is a finite maximal antichain in $\Phi^*$. By Lemma \ref{g4} and H\"{o}lder's inequality,
\begin{eqnarray*}
&&\phi_{k,r}^{\frac{r}{\xi_r}}e_{\phi_{k,r},r}^r(\mu)\geq d_5\phi_{k,r}^{\frac{r}{\xi_r}}\sum_{\rho\in\Gamma_{k,r}(\theta)}t_\rho c_\rho^r\\&&\geq d_5\phi_{k,r}^{\frac{r}{\xi_r}}\bigg(\sum_{\rho\in\Gamma_{k,r}(\theta)}(t_\rho c_\rho^r)^{\frac{\xi_r}{\xi_r+r}}\bigg)^{\frac{\xi_r+r}{\xi_r}}M_{k,r}^{-\frac{r}{\xi_r}}(\theta)\\&&
=d_5\phi_{k,r}^{\frac{r}{\xi_r}}M_{k,r}(\theta)^{-\frac{r}{\xi_r}}\geq d_5>0.
\end{eqnarray*}
If $\xi_r=t_r$, then $\sum_{\sigma\in\Gamma_{k,r}}(p_\sigma c_\sigma^r)^{\frac{\xi_r}{\xi_r+r}}=1$. By Lemma \ref{g4} and H\"{o}lder's inequality,
\begin{eqnarray*}
&&\phi_{k,r}^{\frac{r}{\xi_r}}e_{\phi_{k,r},r}^r(\mu)\geq d_5\phi_{k,r}^{\frac{r}{\xi_r}}\sum_{\rho\in\Gamma_{k,r}}p_\sigma s_\sigma^r \\&&\geq d_5\phi_{k,r}^{\frac{r}{\xi_r}}\bigg(\sum_{\sigma\in\Gamma_{k,r}}(p_\sigma c_\sigma^r)^{\frac{\xi_r}{\xi_r+r}}\bigg)^{\frac{\xi_r+r}{\xi_r}}N_{k,r}^{-\frac{r}{\xi_r}}\\&&
=d_5\phi_{k,r}^{\frac{r}{\xi_r}}N_{k,r}^{-\frac{r}{\xi_r}}\geq d_5>0.
\end{eqnarray*}
By Lemma \ref{g5}, we know that $\phi_{k,r}\leq \phi_{k+1,r}\leq d_1\phi_{k,r}$. Hence, by Lemma 2.4 (b3) of \cite{Zhu:12}, we deduce that $\underline{Q}_r^{\xi_r}(\mu)>0$. As a consequence, we have, $\underline{D}_r(\mu)\geq\xi_r$.

Now let $s>\xi_r$. Then $a(s),b(s)<1$. As we did in Lemma \ref{g6}, one can see
\begin{eqnarray*}
I_k(s)&=&\sum_{\sigma\in\Psi_{k,r}}\sum_{\rho\in\Gamma_{k,r}(\sigma)}(p_\sigma s_\sigma^r t_\rho c_\rho^r)^{\frac{s}{s+r}}+\sum_{\sigma\in\Gamma_{k,r}}(p_\sigma s_\sigma^r)^{\frac{s}{s+r}}\\&<&(1-b(s))^{-1}+1=:c(s).
\end{eqnarray*}
It follows that $\phi_{k,r}\leq c(s)\underline{\eta}_r^{-\frac{(k+1)s}{s+r}}$. Thus, by Lemma \ref{g7},
\begin{eqnarray}\label{s16}
\phi_{k,r}^{\frac{r}{s}}e^r_{\phi_{k,r},r}(\mu)&\leq& \phi_{k,r}^{\frac{r}{s}}\bigg(p_0|C|^r\sum_{\sigma\in\Psi_{k,r}}\sum_{\rho\in\Gamma_{k,r}(\sigma)}p_\sigma s_\sigma^r t_\rho c_\rho^r+\sum_{\sigma\in\Gamma_{k,r}}p_\sigma s_\sigma^r\bigg)\nonumber\\&\leq&p_0|C|^r\phi_{k,r}^{\frac{r}{s}}\phi_{k,r}\underline{\eta}_r^k\leq
p_0|C|^r c(s)^{\frac{s+r}{s}}\underline{\eta}_r^{-1}.
\end{eqnarray}
This implies $\underline{Q}_r^s(\mu)<\infty$ by Lemma 2.4 (b3) of \cite{Zhu:12}. In particular, we have $\overline{D}_r(\mu)\leq s$. By the arbitrariness of $s$, we conclude that $\overline{D}_r(\mu)\leq\xi_r$. This, together with the analysis in the preceding paragraph, yields that $D_r(\mu)=\xi_r$.

Assume that $s_r>t_r$. Using Lemma \ref{g6} and a similar argument as above, one can easily show that $\underline{Q}_r^{\xi_r}(\mu)<\infty$.
If $s_r=t_r$, then by Lemmas \ref{g4}, \ref{g6} and H\"{o}lder's inequality with exponent less than one, we have
\begin{eqnarray*}
&&\phi_{k,r}^{\frac{r}{\xi_r}}e_{\phi_{k,r},r}^r(\mu)\geq d_5\phi_{k,r}^{\frac{r}{\xi_r}}\sum_{\sigma\in\Psi_{k,r}}\sum_{\rho\in\Gamma_{k,r}(\sigma)}p_\sigma s_\sigma^r t_\omega c_\omega^r\\&&\geq d_5\phi_{k,r}^{\frac{r}{\xi_r}}\bigg(\sum_{\sigma\in\Psi_{k,r}}\sum_{\rho\in\Gamma_{k,r}(\sigma)}(p_\sigma s_\sigma^r t_\omega c_\omega^r)^{\frac{\xi_r}{\xi_r+r}}\bigg)^{\frac{\xi_r+r}{\xi_r}}\times\\&&\;\;\;\;\times\bigg(\sum_{\sigma\in\Psi_{k,r}}M_{k,r}(\sigma)\bigg)^{-\frac{r}{\xi_r}}
\geq d_5l_{1,k}^{\frac{\xi_r+r}{\xi_r}}\;\;{\rm by}\;\;(\ref{s33}).
\end{eqnarray*}
It follows by Lemma 2.4 (b3) of \cite{Zhu:12}, that $\underline{Q}_r^{\xi_r}(\mu)=\overline{Q}_r^{\xi_r}(\mu)=\infty$. This completes the proof of the theorem.

\begin{example}{\rm
Let $f_i,i=1,2$ and $g_i,i=1,2$ be as defined below:
\begin{eqnarray*}
f_1(x):=\frac{1}{4}x,\;f_2(x):=\frac{1}{4}x+\frac{3}{4};\\g_1(x):=\frac{1}{8}x+\frac{1}{3},\;g_2(x):=\frac{1}{8}x+\frac{13}{24}.
\end{eqnarray*}
Let $\nu$ be the self-similar measure associated with $(g_1,g_2)$ and a probability vector $(t_1,t_2)$.
Note that $g_i([\frac{1}{3},\frac{2}{3}])\subset[\frac{1}{3},\frac{2}{3}]$ for $i=1,2$. Thus, $C={\rm supp}(\nu)\subset[\frac{1}{3},\frac{2}{3}]$.
Let $(p_i)_{i=0}^2$ be a probability vector with $p_i>0$ for all $i=0,1,2$. Then Theorem \ref{mthm1} holds for the corresponding ISM $\mu$. To see this, it suffices to check the IOSC. Let $U:=(0,1)$. Then we have
\begin{eqnarray*}
&&f_i(U)\subset U,\;g_i(U)\subset U,i=1,2;\\&&f_1(U)\cap f_2(U)=\emptyset,\;g_1(U)\cap g_2(U)=\emptyset.
\end{eqnarray*}
This shows that both $(f_1,f_2)$ and $(g_1,g_2)$ satisfy the OSC and (A1)-(A3) are satisfied. We clearly have that $\nu(\partial(U))=0$. For $i=1,2$, we have,
\begin{equation*}
f_i({\rm cl}(U))\cap C\subset f_i({\rm cl}(U))\cap \big[\frac{1}{3},\frac{2}{3}\big]=\emptyset.
\end{equation*}
Hence, (A4) is satisfied and Theorem \ref{mthm1} holds for $\mu$.

}\end{example}

\end{document}